\documentclass[pdflatex,sn-mathphys-num]{sn-jnl}

\usepackage{graphicx}%
\usepackage{multirow}%
\usepackage{amsmath,amssymb,amsfonts}%
\usepackage{amsthm}%
\usepackage{mathrsfs,natbib}%
\usepackage[title]{appendix}%
\usepackage{xcolor}%
\usepackage{textcomp}%
\usepackage{manyfoot}%
\usepackage{booktabs}%
\usepackage{algorithm}%
\usepackage{algorithmicx}%
\usepackage{algpseudocode}%
\usepackage{listings}%
\usepackage{subcaption}
\usepackage{setspace}

\theoremstyle{thmstyleone}%
\newtheorem{theorem}{Theorem}
\newtheorem{proposition}[theorem]{Proposition}%
\newtheorem{lemma}{Lemma}
\theoremstyle{thmstyletwo}%
\newtheorem{example}{Example}%
\newtheorem{remark}{Remark}%

\theoremstyle{thmstylethree}%
\newcommand{\be}{\begin{equation}}
\newcommand{\ee}{\end{equation}}
\newcommand{\bes}{\begin{equation*}}
\newcommand{\ees}{\end{equation*}}
\newcommand{\bal}{\begin{aligned}}
\newcommand{\eal}{\end{aligned}}
\newcommand{\ba}{\begin{eqnarray}}
\newcommand{\ea}{\end{eqnarray}}
\newcommand{\bas}{\begin{eqnarray*}}
\newcommand{\eas}{\end{eqnarray*}}

\def\D{{\mathcal D}}
\def\P{{\mathcal P}}
\def\N{{\mathcal N}}

\def\E{{\mathcal E}}
\def\F{{\mathcal F}}

\def\W{{\mathcal W}}

\def\Re{\mathbb{R}}

\def\bfb{{\bf b}}

\def\bfx{{\bf x}}
\def\bfy{{\bf y}}
\def\bfz{{\bf z}}
\def\bfw{{\bf w}}

\def\bfbeta{{\bf {\boldsymbol \beta}}}

\def\bfmu{{\bf {\boldsymbol \mu}}}
\def\bfone{{\bf 1}}

\def\diag{\mbox{diag}}

\def\ef{{ef}}
\def\efl{{efl}}

\def\mvp{{mvp}}
\def\r{{r}}

\def\rmv{{rmv}}
\def\rmvl{{rmvl}}
\def\sharpe{{sharpe}}
\def\target{{target}}

\def\hatbeta{\widehat{\bfbeta}}
\def\hatC{\widehat{C}}

\begin{document}

\title[Portfolio Selection by Joint Estimation-Robustness Optimization with the Return and Risk Constraints]{Portfolio Selection by Joint Estimation-Robustness Optimization with the Return and Risk Constraints}


\author[1]{\fnm{Yingxiao} \sur{Wang}}\email{yingxiaowang@pku.edu.cn}

\author[2]{\fnm{Lingchen} \sur{Kong}}\email{lchkong@bjtu.edu.cn}

\author[3]{\fnm{Selin D.} \sur{Ahipasaoglu}}\email{sda1u20@soton.ac.uk}

\author*[4]{\fnm{Houduo} \sur{Qi}}\email{houduo.qi@polyu.edu.hk}

\affil[1]{\orgdiv{Academy for Advanced Interdisciplinary Studies}, \orgname{Peking University}, \orgaddress{\city{Beijing}, \postcode{100871}, \country{China}}}

\affil[2]{\orgdiv{School of Mathematics and Statistics}, \orgname{Beijing Jiaotong University}, \orgaddress{\city{Beijing}, \postcode{100044}, \country{China}}}

\affil[3]{\orgdiv{School of Mathematical Sciences}, \orgname{University of Southampton}, \orgaddress{\city{Southampton}, \postcode{SO17 1BJ}, \country{United Kingdom}}}

\affil[4]{\orgdiv{Department of Data Science and Artificial Intelligence}, \orgname{The Hong Kong Polytechnic University}, \orgaddress{\city{Hongkong}, \postcode{999077}, \country{China}}}


\abstract{In answering the main criticism of the classical robust optimization that the size of the uncertainty set has to be determined beforehand, a recently proposed Joint Estimation and Robustness Optimization framework treats the size of the uncertainty set as a variable and maximizes it subject to some robust conditions over the uncertainty set. This paper studies its effect on portfolio selection when both the robust return and risk constraints are built into Joint Estimation and Robustness Optimization framework. Although this framework does not directly pursuit any optimality of portfolio return or risk, we report a surprising result that Joint Estimation and Robustness Optimization framework actually selects efficient portfolios. This is proved by reformulating the robust risk constraint as a tractable form through the S-procedure. Moreover, a closed-form formula was derived when there is only the budget constraint. An efficient algorithm is proposed when more constraints are added. Numerical experiments illustrate the validity of the reported results. This research paves the way for Joint Estimation and Robustness Optimization framework to be applied to the portfolio selection problem with return and risk related robust constraints.}




\maketitle

\section{Introduction}\label{sec1}

This paper is primarily concerned with the two issues that have to be resolved for constructing a personalized portfolio by an investor. The first issue is about handling the uncertainty among the historical financial data.
The uncertainty is often related to the sensitivity of the constructed portfolio on some estimated quantities such as the expected return and the covariance matrix, see \cite{ledoit2004a, ledoit2017nonlinear} and the references therein. Robust optimization is often used to take the uncertainty \citep{ben2009robust}. The second issue is about setting up the goals that an investor would wish to achieve. Those may include a minimal level of return of his portfolio, maximal level of the risk he may face, a reasonable Sharpe ratio, and/or a certain level of diversification. Those goals each has become a well studied topic in portfolio optimization and achieving some of those goals altogether often leads to a multi-objective optimization problem. In particular, the balancing between the return and the risk leads to Markowitz's seminal work on the Mean-Variance (MV) model \citep{markowitz1952}. This paper follows the robust optimization framework recently introduced by \cite{Zhu2022joint}. We investigate the effect of the mean and variance of a portfolio under this framework and establish the basic results related to the efficient frontier of the MV model. 

\subsection{Robust optimization}

Since the publication of the book \cite{ben2009robust}, Robust Optimization (RO) has become a well-established
framework to achieve the best decisions under uncertainty.
Three elements are essential in building an effective RO formulation: the uncertainty set, the objective function, and the effective algorithms. 
Defining the uncertainty set encompasses all the possible values of the quantities of interest (e.g., the expected return) with a certain degree of deviation from the true values under a model of uncertainty. 
There are two types of uncertainty sets that are often studied: ellipsoidal and box-constrained.
In the field of robust portfolio optimization, there 
exists extensive research to address the first two elements, see e.g., \citep{Ghaoui1997robust, Ghaoui1998robust, ben1998robust, ben1999robust, goldfarb2003robust, tutuncu2004robust, fabozzi2007robust, Ling2020robust,Chakrabarti2024parameter}.
The third element is about developing effective algorithms for the robust model, obtained by specifying the first two. 
In this regard, the second-order cone programming (SOCP) and semidefinite programming (SDP) are
often used. There exists a large number of examples in the book \cite{ben2009robust} that
reformulate RO into tractable formulations through SOCP and/or SDP. Similar techniques using SOCP relaxations can be found in \cite{xu2023} and \cite{luo2025}.
The techniques used for reformulation are diverse and often depend on the applications at hand.
The bibliographic overviews of this area for financial portfolios can be found in \cite{Xidonas2020robust} and \cite{Salo2024fifty}.

One criticism of RO is that 
the size of the uncertainty set has to be decided beforehand in order for RO to be well-defined.
How well robust optimization techniques perform relies heavily on the definition and size of the uncertainty set. 
Indeed, as highlighted by \cite{Bertsimas2018data}, a poorly defined or sized uncertainty set often leads to an overly conservative or computationally intractable portfolio model.

In answering to this criticism, \cite{Zhu2022joint} proposed a new framework of RO.
The proposal is based on the new interpretation that
 the size of uncertainty set of an estimate is intrinsically  related to
the confidence interval of the estimate.
In other words, the uncertainty set contains
all the values of the estimate that cannot be rejected given a confidence level.
Ideally, one would explore all those possible values. This leads to maximizing the size of
the uncertainty set. 
Since an arbitrary value from this uncertainty set is as good as any other, if one value satisfies a condition, all other values from this uncertainty set must satisfy
the same condition. This aspect of the requirement gives rise to the robustness and results in the framework of
JERO (Joint Estimation and Robustness Optimization).
JERO has found a number of important applications and portfolio selection is one of them.


\subsection{JERO framework}

We briefly describe this framework using a minimal set of mathematical notation.
For extensive coverage of this framework in a more general modeling setting, please refer to \cite{Zhu2022joint}.
Suppose we have observed data set $\D = \left\{ \bfx_i\right\}_{i=1}^N$ with each $\bfx_i \in \Re^n$.
Let $\bfbeta \in \Re^m$ be the parameter of interest to be estimated  (e.g., being the expected return vector of $n$ assets). Let $\rho(\bfbeta; \D)$ 
be a metric that is used for estimation and for 
defining the uncertainty set with the size $r\ge 0$ and a known $\widehat{\rho} \ge 0$:
\begin{equation}\label{Uncertainty_Set}
    \E (r; \D) := \left\{  \bfbeta\ | \  \rho(\bfbeta; \D) \le r + \widehat{\rho} \right\} .
\end{equation}
Here ``$:=$'' means ``define''.
It is often the case that there is known $\widehat{\bfbeta}$ (e.g., the maximum likelihood estimate)
such that $\widehat{\rho} = \rho (\widehat{\bfbeta}; \D )$.
In this way, the uncertainty set $\E(r; \D)$ is not empty. 
A key component in JERO is its robust constraints.
They are constructed as follows. 
Suppose there exists a function $a(\bfx, \bfbeta)$ of variables $\bfx$ and $\bfbeta$ such that the constraint $a(\bfx, \bfbeta) \le \tau$ is satisfied for given $\tau$.
The robust version is for this constraint to be satisfied for all $\bfbeta$ in the uncertainty set $\E(r; \D)$:
\[
  \max\{  a(\bfx, \bfbeta ) \ | \  \bfbeta \in \E(r; \D)\} \le \tau .
\]

The JERO in its general form is given by 
\be \label{JERO} 
\begin{array}{ll}
	\max\limits_{\bfx,\; r} & r \\
	\mbox{s.t.} & \max\{  a_i(\bfx, \bfbeta ) \ |  \ \bfbeta \in \E(r; \D)\} \le \tau_i ,  \ i = 1, \cdots, k,\\
	&  \bfx \in \P , 
\end{array} 
\ee 
where we have $k$ robust constraints and $\P$ is a closed convex set for $\bfx$. 
This is the main modeling contribution from \cite{Zhu2022joint}.

We note that JERO is in the format of bi-level optimization and is in general computationally intractable \cite{dempe2020bilevel}.
Under the assumption that each $a_i(\bfx, \bfbeta)$ is a saddle function (i.e., $a_i(\bfx, \bfbeta)$ is convex in $\bfx$ and concave in $\bfbeta$),
the model (\ref{JERO}) is often tractable (e.g., solved by convex programming) via the convex duality (Fenchel duality).
A number of examples were given in \cite{Zhu2022joint} where (\ref{JERO}) can be solved through a sequence of second-order cone programs. 
This represents the algorithmic contribution from \cite{Zhu2022joint}.
In particular, Example 2 in \cite{Zhu2022joint} demonstrates its application to
robust portfolio selection when $k=1$ and $a(\bfx, \bfbeta)$ is set to  be the negative (expected) return of a portfolio: $a(\bfw, \bfbeta) = -\bfbeta^\top \bfw$. Here we use $\bfw$ instead of $\bfx$ to denote the portfolio weights. 
In this way, the robust constraint
in (\ref{JERO}) means that the desired portfolio $\bfw$ achieves the return target $\tau$ for any $\bfbeta$ from the largest uncertainty set.
The robust portfolio is tractable because  $a(\bfw, \bfbeta)$ is linear in $\bfbeta$ (hence concave as well). 
It is further proved that the selected robust portfolio corresponds to the maximum Sharpe ratio portfolio.

As evidenced in the Markowitz MV model, return and risk are two of the most valuable measurements for a portfolio.
Therefore, one important question regarding JERO is to understand what the robust portfolio would achieve when both the robust return and robust risk constraints are built
into the JERO model ($k=2$).
It turns out that it is not an easy question to answer
because the risk of a portfolio, as we will see, is not concave in $\bfbeta$.
This means that the duality approach in 
\cite{Zhu2022joint} cannot be extended to address the robust risk constraint.
A new approach must be found. 
We summarize our main contributions below.

\subsection{Main contributions}

It has become clear that the main model we are going to consider is JERO with both robust return and risk constraints ($k=2$ in (\ref{JERO})).
We call it JERO-MV model and report the following findings.

\begin{itemize}
	\item[(i)] {\bf Robust risk constraint has a convex representation}.
	In contrast to the Fenchel duality approach in \cite{Zhu2022joint}, 
	the celebrated S-Lemma/S-Procedure in SDP (see \citep[Appendix B]{ben2009robust} and \citep[Appendix B]{boyd2004convex}) 
	is used to derive the convex representation.
	Together with the convex representation of the robust return constraint derived in \cite{Zhu2022joint}, this result
	leads to a tractable reformulation of the JERO-MV model.

	\item[(ii)] {\bf JERO-MV portfolio is efficient}.
	This is surprising because JERO-MV model does not directly pursuit any optimality of return or risk (e.g., maximum return or minimum risk),
	but instead it considers all portfolios that meet some given targets (i.e, portfolio return is not less than $\mu_\target$ and portfolio risk
	is not bigger than $\sigma^2_\target$.) 
	Hence, JERO-MV model appears more conservative than the Markowitz MV model and one would expect it choose a portfolio that is 
	less efficient.
	Our result shows that the JERO-MV acts like a utility function that chooses an efficient portfolio.
	Moreover, the location of the efficient portfolio can be precisely characterized. For example,
	for the case when $\mu_\target$ is set bigger than the minimum variance portfolio return $\mu_\mvp$,
	the efficient portfolio can be computed through a quadratic formula in Theorem \ref{Thm-EF-A} and is illustrated in Fig.~\ref{Fig-largeTarget}. 
	For the case $\mu_\target < \mu_\mvp$ (say, in the scenario all portfolios should have return above that of a risk-free asset),
	we also have similar characterizations, see Theorem \ref{Thm-EF-B} and its illustration in Fig.~\ref{Fig-smallTarget}. 
	Those results are obtained when there is only one extra constraint: the budget constraint.

	\item[(iii)] {\bf JERO-MV model can be efficiently solved}.
	When more constraints are added, there is certainly a feasibility issue whether there exists a portfolio that meets both the robust return and
	risk constraints. To alleviate this issue, we optimize the violation of robust return and employ the bisection method proposed in \cite{Zhu2022joint}.
	This way, we can solve a sequence of feasible problems. Depending on the final objective value, we can decide whether the original JERO-MV is feasible or not.
	If it is feasible, our method returns the robust portfolio. Otherwise, we need to re-set the targets $\mu_\target$ and $\sigma_\target$ so that the new JERO-MV is feasible. Numerical results are used to demonstrate the efficiency of the proposed algorithm.

\end{itemize}

The paper is organized as follows.
In the next section, we describe some mathematical notation used and list some key results about the efficient frontier for
the Markowitz MV model.
Section~\ref{Section-JERO-MV} introduces the JERO-MV model and reformulates the robust risk constraint in a tractable form, see Proposition~\ref{Thm-variance}.
In Section~\ref{Section-EF}, we prove that the optimal JERO-MV portfolio is efficient and its location on the efficient frontier is characterized in Theorem~\ref{Thm-EF-A} and Theorem~\ref{Thm-EF-B}
when there is only the budget constraint.
We also propose a new bisection algorithm when more constraints are included. 
Numerical experiments are reported in Section~\ref{Section-NE}. 
We conclude the paper in Section~\ref{Section-Conclusion}.
Some technical results are moved to the appendices.

\section{Background on Portfolio Selection} \label{Section-Background}

We first set up the notation used in this paper. The boldfaced small letter such as $\bfw \in \Re^n$ denotes a column vector
of length $n$. Its transpose is denoted by $\bfw^\top$. Therefore, $\bfx^\top \bfy$ means the standard inner product between
two vectors of the same length. Let $\mathbb{S}^n_+$ (resp. $\mathbb{S}^n_{++}$) denote the set of all positive semidefinite (resp. definite)
$n \times n$ symmetric matrices. We often write $A \succeq 0$ for $A \in \mathbb{S}^n_+$ and $A \succeq B$ means $(A-B) \in \mathbb{S}^n_+$.
For a matrix $A \in \mathbb{S}^n_{++}$, the $A$-weighted norm of $\bfx \in \Re^n$ is defined by
$
 \| \bfx\|_A := \sqrt{\bfx^\top A \bfx} .
$
For $A \in \mathbb{S}^n_+$, $A^{1/2} \in \mathbb{S}^n_+$ is the square root of $A$.


We now introduce the portfolio selection problem.
Suppose there are $n$ securities with their returns at time $t$ denoted by
$
\bfx_t^\top = [x_{1t}, \cdots, x_{nt}].
$
A portfolio consisting of those $n$ assets is denoted by a vector $\bfw \in \Re^n$ with each weight $w_i$ representing 
the amount percentage of the capital to be invested in each asset.
Therefore, we have $\bfone_n^\top \bfw = 1$ (known as the budget constraint), where $\bfone_n$ is the column vector of all ones of size $n$.  
A portfolio is constructed through historical data $\D = \{ \bfx_t\}_{t=1}^N$, where it is assumed that the data points are independent and identically distributed (i.i.d.). 
Let $\bfbeta \in \Re^n$ be the expected return vector of $\bfx_t$ (i.e., $\bfbeta = E[\bfx_t]$) and
$C \in \mathbb{S}^n_+$ be its covariance matrix. 
We assume that $\bfx_t$ follows a multivariate normal distribution $\N(\bfbeta, C)$.
Let $\hatbeta$ and $\hatC$ be respectively the maximum likelihood estimates of $\bfbeta$ and $C$ from the data set $\D$ \citep{mardia1979multivariate}:
\[
\hatbeta = \frac 1N \sum_{t=1}^N \bfx_t \quad \mbox{and} \quad
\hatC = \frac 1N \sum_{t=1}^N (\bfx_t - \hatbeta) ( \bfx_t - \hatbeta)^\top .
\]
We assume $\hatC$ is positive definite. 
For a portfolio $\bfw$, its variance is denoted by $\sigma_\bfw^2 = \bfw^\top \hatC \bfw$. 
Throughout the paper, the two names of variance and risk of a portfolio are used interchangeably.  
%

In \cite{markowitz1952}, authors proposed the concept of efficient portfolio, i.e., the portfolio with minimum variance for given return or less and maximum return for given variance or less. 
The set of all efficient portfolios is called the efficient frontier.
We briefly review some quantitative results about the efficient frontier for the Markowitz mean-variance model defined by $\hatbeta$ and $\hatC$:
\be \label{Markowitz-MV}
  \min \ \bfw^\top \hatC \bfw , \qquad \mbox{s.t.} \quad \hatbeta^\top \bfw = \mu, \ \ \bfone_n^\top \bfw = 1,
\ee
where $\mu$ is a given return level.
In particular, the minimum variance portfolio $\bfw_\mvp$ plays an important role in our analysis:
\[
  \bfw_\mvp := \arg\min_{\bfw} \left\{
    \bfw^\top \hatC \bfw \ \Big\vert \ \bfone_n^\top \bfw = 1
  \right\}
  = \frac{ \hatC^{-1} \bfone_n }{ \bfone_n^\top \hatC^{-1} \bfone_n } .
\]
Its expected return is $\mu_\mvp = \hatbeta^\top \bfw_\mvp$. 
It is known that the optimal solutions of (\ref{Markowitz-MV}) for all $\mu \ge \mu_\mvp$ form the efficient frontier of all feasible 
portfolios. 
In order to describe the efficient frontier, we 
define the so-called self-financed portfolio $\bfw_0$ (due to $\bfone_n^\top \bfw_0=0$) \citep{korkie2002mean}:
\[
  \bfw_0 := \frac{\hatC^{-1} (\hatbeta - \mu_\mvp \bfone_n ) }{ \| \hatbeta - \mu_\mvp \bfone_n \|_{\hatC^{-1} }}, 
  \quad \mbox{with} \ \
  \mu_\mvp = \frac{\bfone_m^\top \hatC^{-1} \hatbeta}{\bfone_n^\top \hatC^{-1} \bfone_n} . 
\]
The result below says that the efficient frontier can be obtained through combinations of $\bfw_\mvp$ and $\bfw_0$.

\begin{proposition} \label{Prop-EF}
For a given $\gamma \ge 0$, define the portfolio $\bfw_\gamma$ by
\[
  \bfw_\gamma := \bfw_\mvp + \gamma \bfw_0 .
\]	
Its return and variances are defined by $\mu_\gamma := \hatbeta^\top \bfw_\gamma$ and $\sigma_\gamma^2 := \bfw_\gamma^\top \hatC \bfw_\gamma$.
The following hold.
	
\begin{itemize}
	\item[(i)] The efficient frontier of (\ref{Markowitz-MV}) for $\mu \ge \mu_\mvp$ is given by \cite[Definition~4.2.1]{brugiere2020quantitative}:
	\[
	   \F^+ := \left\{ \bfw_\gamma \ | \  \gamma \ge 0 \right\}.
	\]
	
	\item[(ii)] Let $\mu_0 := \hatbeta^\top \bfw_0$. It holds \cite[Proposition~4.1, Corollary~4.2.1]{brugiere2020quantitative}:
	\bes
      \bal
          &\sigma_\gamma^2 = \sigma_\mvp^2 + \gamma^2 =
	  \sigma_\mvp^2 + \left(
	  \frac{ \mu_\gamma - \mu_\mvp }{ \mu_0 }
	 \right)^2 \quad \mbox{and} \quad \mu_0 = \| \hatbeta - \mu_\mvp \bfone_n \|_{\hatC^{-1}} .
      \eal
	\ees
	
\end{itemize}	
	
\end{proposition}

\section{JERO-MV Model} \label{Section-JERO-MV} 

We describe the main model in this section. We begin with the robust return constraint studied in \cite{Zhu2022joint}.

\subsection{Robust return constraint}

 The estimation metric is defined by
\begin{equation*}
	\begin{aligned}
    \rho(\boldsymbol{\beta};\D)=  \min\limits_{C \in \mathbb{S}_+^n}\frac{1}{2}\left( \ln\det C  + \frac{1}{N}\sum\limits_{t=1}^N \| {\bfx}_t-\boldsymbol{\beta} \| ^2_{C^{-1}}  \right).
	\end{aligned}
\end{equation*}
The optimal solution, denote by $C_{\bfbeta}$ is given by \cite[Proof of Proposition 6]{Zhu2022joint}
\be \label{Cbeta}
  C_{\bfbeta} = \hatC + (\bfbeta - \hatbeta) (\bfbeta - \hatbeta)^\top
\ee
and $\rho(\bfbeta; \D)$ can be computed by
\begin{equation}\label{eq2.5}
	\begin{aligned}
		\rho(\boldsymbol{\beta};\mathcal{D}) = \ \frac{1}{2}\left( n+\ln\det \hatC  + \ln( 1+ \| \boldsymbol{\beta}-\hatbeta \| ^2_{\hatC^{-1}}) \right).
	\end{aligned}
\end{equation}
Thus, the estimation uncertainty set (\ref{Uncertainty_Set}) with $\widehat{\rho} = \rho(\hatbeta; \D)$ can be written as
\begin{equation}
	\begin{aligned}
		\mathcal{E}(r;\mathcal{D})=\left\{\boldsymbol{\beta}\in \Re^n \ \big\vert \ \Vert \boldsymbol{\beta}-\hatbeta \Vert_{\hatC^{-1}}^2 \leq e^{2r}-1 \right\}.
	\end{aligned}
\end{equation}
It is closed and convex. 

Let $\mu_\target$ be a given return target,  $r\ge 0$ and $\bfw$ be a given portfolio. 
The robust return constraint requires the returns of this portfolio computed from any return vector $\bfbeta$ in the uncertainty set to be not less 
than the return target:
\be \label{Robust-Return-Constraint}
 \min\{  \bfbeta^\top \bfw \ | \ \bfbeta   \in \E(r; \D)  \} \ge \mu_\target.
\ee 
This robust constraint has a tractable convex representation via the Fenchel duality theory:
\be \label{SOCP-Mean}
- \hatbeta^\top \bfw + \sqrt{e^{2r}-1} \| \hatC^{1/2} \bfw \|_2 + \mu_\target \le 0 . 
\ee 
In other words, the robust return constraint can be reformulated as a second-order cone constraint when $r$ is fixed.

The JERO with the robust return constraint \cite[Example 2]{Zhu2022joint} is defined by
\be \label{JERO-M}
  \mbox{(JERO-M)} \quad  \left\{  \begin{array}{ll}
   	 \max\limits_{\bfw, \; r} & r \\
   	 \mbox{s.t.} & 
   	               \mbox{Constraint (\ref{SOCP-Mean})} \\
   	             & \bfone_n^\top \bfw = 1, \ \bfw\ge 0, \ r \ge 0 .
   \end{array} 
   \right .
\ee 
With the variance of the portfolio $\bfw$ given by $\sigma_\bfw^2 := \bfw^\top \hatC \bfw$, it is easy to see that the constraint (\ref{SOCP-Mean}) is equivalent to
\[
 \sqrt{e^{2r}-1} \le S(\bfw) := \frac{ \hatbeta^\top \bfw - \mu_{target} }{ \sigma_\bfw }.  \ (\mbox{Sharpe ratio})
\]
Hence, the optimal portfolio by (\ref{JERO-M}) is the maximum Sharpe-ratio portfolio.
It is important to note that the constraint $\bfw \ge 0$ (i.e., the long-only constraint) is included in
(\ref{JERO-M}). It makes the feasible portfolios bounded and hence an optimal solution  must exist.
Without this long-only constraint, the optimization problem (\ref{JERO-M}) may not have an optimal solution.
For example, when $\mu_\target > \mu_\mvp$, (\ref{JERO-M}) would not have an optimal solution if the long-only constraint is removed.
\subsection{Robust risk constraint}

In the theory of mean-variance model of \cite{markowitz1952}, the variance of a portfolio is one of the most
important factors in constructing efficient portfolios. 
Given $\bfbeta$, $C_{\bfbeta}$ in (\ref{Cbeta}) is the maximum likelihood estimate for the covariance matrix $C$.
The associated risk, denoted by $a_\sigma(\bfw, \bfbeta)$, is then given by
\bes
 \bal
     a_\sigma(\bfw, \bfbeta) := \ & \bfw^\top C_{\bfbeta} \bfw 
  = \ \bfw^\top \hatC \bfw + (\bfbeta - \hatbeta)^\top \bfw \bfw^\top (\bfbeta - \hatbeta).
 \eal
\ees
This risk can also be interpreted as the empirical risk proposed by Brodie et. al. \cite[Equation~(1)]{Brodie2009Sparse}:
\[
  \sigma_e^2 (\bfw, \bfbeta) := \frac 1N \| (\bfw^\top \bfbeta) \bfone_N - X \bfw\|^2, 
\]
where $X^\top := [\bfx_1, \ldots, \bfx_N]$ is the sampled data matrix. 
Using the definitions of $\hatbeta$ and $\hatC$, one can prove that $ \sigma_e^2 (\bfw, \bfbeta) = a_\sigma(\bfw, \bfbeta)$. 

Suppose we are given a risk level $\sigma_\target^2$ and a portfolio $\bfw$. We wish that the risk of this portfolio based on the covariance matrix $C_{\bfbeta}$
is no greater than the target for all $\bfbeta$ from the uncertainty set:
\be \label{Robust-variance-constraint}
  \max  \left\{  a_\sigma (\bfw, \bfbeta) \ | \  \bfbeta \in \E(r, \D) \right\} \le \sigma_\target^2  .
\ee
It is easy to see that the function $a_\sigma(\bfw, \bfbeta)$ is convex in $\bfw$, and convex in $\bfbeta$. 
It is not concave in $\bfbeta$, breaking the assumption used in \cite{Zhu2022joint}. 
We have the following characterization of (\ref{Robust-variance-constraint}), whose proof is based on the S-procedure and is moved to the appendix.

\begin{proposition} \label{Thm-variance}
The robust constraint (\ref{Robust-variance-constraint}) holds if and only if 
\be \label{convex-variance}
 e^{2r} \Big( \bfw^\top \hatC \bfw \Big) \le \sigma_\target^2.
\ee 
	
\end{proposition}

This leads to the JERO-MV model with both the robust return constraint (\ref{Robust-Return-Constraint}) and the robust risk constraint (\ref{Robust-variance-constraint}).
We use the equivalent reformulations (\ref{SOCP-Mean}) and (\ref{convex-variance}) for the two constraints. JERO-MV model is as follows:
\be \label{JERO-MV}
\begin{array}{ll}
	\max\limits_{r, \bfw} & r \\
	\mbox{s.t.} & \sqrt{ e^{2r} - 1} \|  \hatC^{1/2} \bfw \| \le   \hatbeta^\top \bfw - \mu_\target \\ [1ex]
	&  e^{2r}  \|  \hatC^{1/2} \bfw \|^2 \le  \sigma_\target^2 \\ [1ex]
	& \bfone_n^\top \bfw = 1, \ \ r \ge 0 . 
\end{array} 
\ee  
We will first address JERO-MV, where only budget constraint $\bfone_n^\top \bfw =1$ is included. 
Later on, we will include more constraints.

\section{Efficient Portfolio by JERO-MV} \label{Section-EF}

The main result is in this part is that JERO-MV (\ref{JERO-MV}) selects an efficient portfolio and its solution is unique.
Moreover, we characterize the location of the optimal portfolio on the efficient frontier.
It turns out that the location depends on the condition whether the return target $\mu_\target$ is bigger than $\mu_\mvp$ or not.
The two cases $\mu_\target \ge \mu_\mvp$ and $\mu_\target < \mu_\mvp$ are considered  separately.

\subsection{The case: $\mu_\target \ge \mu_\mvp$} \label{subEF1}

There are three efficient portfolios that play an important role here.
The first one is the efficient portfolio that realizes the return target $\mu_\target$.
We name it $\bfw_{\mu\target}$. 
The corresponding return and standard deviation on the efficient frontier is denoted by
$(\mu_\target, \sigma_{\bfmu\target})$ (Point B in Fig.~\ref{Fig-largeTarget}).
The second portfolio realizes the robust risk target.
We call it $\bfw_{\sigma\target}$.
Its corresponding return and standard deviation on the efficient frontier is denoted by
$(\mu_{\sigma\target}, \sigma_\target)$ (Point C in Fig.~\ref{Fig-largeTarget}).
Please note the notation used here: $\mu_\target$ and $\sigma_\target$ are two parameters the user has to specify for JERO-MV model (\ref{JERO-MV}). 
The third portfolio is the maximum Sharpe-ratio portfolio (simply referred as Sharpe portfolio) with respect to the
given return target $\mu_\target$ and is defined by
\be \label{SharpePortfolio}
  \bfw_\sharpe = \arg\max_{\bfw} \left\{
   \frac{\hatbeta^\top \bfw - \mu_\target}{ \sigma_{\bfw}} \ \ \Big\vert \
   	\bfone_n^\top \bfw = 1, \ 
   	\sigma_\bfw \le \sigma_\target
  \right\} .
\ee
Note that we enforced the constraint $\sigma_\bfw \le \sigma_\target$. 
The reason is that, without it, the Sharpe portfolio would not exist unless $\mu_\target < \mu_\mvp$, which is considered in the next case.
The expected return and standard deviation of Sharpe portfolio is denoted by
$(\bfmu_\sharpe, \sigma_\sharpe)$. 
In order for JERO-MV to be feasible, we must set
$
\sigma_\target \ge \sigma_{\mu\target} .
$
In other words, Point C must be above Point B on Fig.~\ref{Fig-largeTarget}.
Our first result shows that Sharpe portfolio is $\bfw_{\sigma\target}$.

\begin{proposition} \label{Prop-Sharpe}
Suppose $\mu_\target \ge \mu_\mvp$ and $\sigma_\target \ge \sigma_{\mu\target}$. We have
$
 \bfw_\sharpe = \bfw_{\sigma\target} .
$ 
\end{proposition}

To facilitate our discussion, for a given portfolio $\bfw$, define
\begin{equation*}
    \begin{aligned}
        r_m(\bfw) := \frac 12 \ln \Big(
  1 +  (\hatbeta^\top \bfw - \mu_\target)^2 / \sigma_{\bfw}^2
  \Big) \quad \mbox{and} \quad r_v(\bfw) := \frac 12 \ln \Big(
  \sigma_\target^2 / \sigma_{\bfw}^2
  \Big) .
    \end{aligned}
\end{equation*}

In fact, $ r_m(\bfw)$ is the value of $r$ when the robust return constraint in JERO-MV model (\ref{JERO-MV}) is binding (i.e., satisfied as equality).
Similarly,  $r_v(\bfw)$ is the value of $r$ when the robust risk constraint is binding.
Since the exponential function $e^{2r}$ is increasing in $r$, the optimal solution $r_*$ for JERO-MV model (\ref{JERO-MV}) is given by
\be \label{Optimal-r}
 \bal
    r_* = \max_{\bfw} \left\{
    \min\{ r_m(\bfw), \  r_v(\bfw)\} \ | \ \mbox{$\bfw$ feasible to (\ref{JERO-MV})} \right\} . 
 \eal
\ee
Any portfolio achieving $r_*$ is an optimal portfolio, denoted by $\bfw_*$.
Before we move on, we emphasize that any feasible portfolio to JERO-MV model (\ref{JERO-MV}) must have its variance falling into 
the interval $[\sigma_{\mu\target}^2, \ \sigma_\target^2]$.
For $\sigma \in [\sigma_{\mu\target}, \ \sigma_\target]$, 
we consider all the portfolios having the same variance $\sigma^2$:
\[
\W_\sigma := \left\{
\bfw \ \Big| \ \sigma_\bfw = \sigma, \ \ \bfw \ \mbox{is feasible}
\right\}.
\]
We have the following results.

\begin{proposition} \label{Prop-r-mv}
Let $\sigma, \sigma_1, \sigma_2 \in [\sigma_{\mu\target}, \ \sigma_\target]$ with $\sigma_1 \le \sigma_2$.
The following statements hold.

\begin{itemize}
	\item[(i)] We have
	\[
	  r_v(\bfw_1) \ge r_v (\bfw_2), \quad \forall\ \bfw_i \in \W_{\sigma_i}, \ \ i=1,2 .
	\]
    Moreover, $r_v(\bfw_{\sigma\target}) = 0$.
	
	\item[(ii)] Let $\bfw_\ef^\sigma \in \W_\sigma$ be the efficient portfolio whose variance is $\sigma^2$. 
	We have
	\bes
	    \bal
	   & r_m (\bfw^\sigma_\ef) \ge r_m (\bfw) , \quad \forall \ \bfw \in \W_\sigma \quad \mbox{and} \quad r_m(\bfw_\ef^{\sigma_1}) \le r_m (\bfw_\ef^{\sigma_2}) .
		\eal
	\ees
	Moreover, $r_m (\bfw_{\mu\target}) = 0$.
	
	\item[(iii)] There exists a unique efficient portfolio $\bfw_\ef$ such that the optimal objective $r_*$ satisfies 
	\be \label{Equilibrium}
	   r_* = r_m(\bfw_\ef) = r_v (\bfw_\ef) .
	\ee
	
\end{itemize}
	
\end{proposition}

For understanding the behaviour of the JERO-MV portfolio, it is important to emphasize that
the two quantities $r_m(\bfw)$ and $r_v(\bfw)$ also become equal at other portfolios. 
However, the meeting point on the efficient frontier by (\ref{Equilibrium}) gives the largest value $r_*$ because
$r_m(\bfw_\ef)$ on the efficient frontier is greater than any other portfolios that has the same variance as that of $\bfw_\ef$. Therefore, maximizing $r$ is to select a portfolio from the efficient frontier.

After those preparations, we are ready to characterize the optimal efficient portfolio, which, together with several other
portfolios,  is plotted in in Fig.~\ref{Fig-largeTarget}.

\begin{theorem} \label{Thm-EF-A}
Suppose $\mu_\target \ge \mu_\mvp$ and $\sigma_\target \ge \sigma_{\mu\target}$.
	The optimal portfolio $\bfw_\rmv$ of JERO-MV (\ref{JERO-MV}) is efficient and is given by
\be \label{Optimal_solution}
   \bfw_{rmv} = \bfw_{mvp} + \gamma \bfw_0  
\ee
with $\gamma := \left\{ \mu_0 (\mu_{target} - \mu_{mvp}) + \sqrt{\Delta} \right\} / {(1+\mu_0^2)}$
and
$
  \Delta := \sigma_{target}^2 - \sigma_{mvp}^2 + (\mu_{\sigma target} - \mu_{target}) (\mu_{\sigma target} + \mu_{target} - 2\mu_{mvp})
$.
\end{theorem}

\begin{figure}[h!]
	\centering
	\includegraphics[width= 0.8\textwidth, height = 0.5\textwidth]{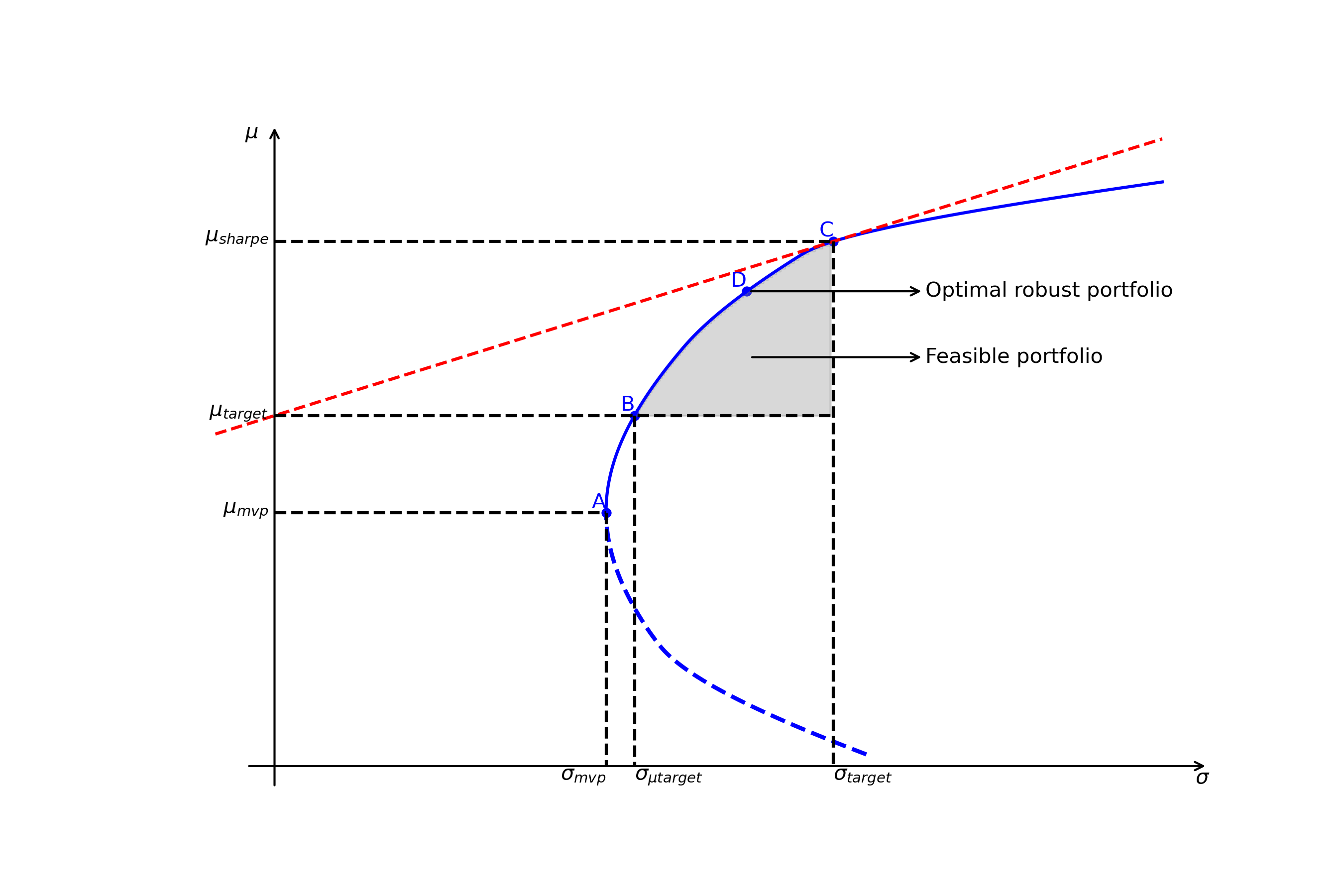}
	\caption{The case: $\mu_\target \ge \mu_\mvp$.
	Point A is the global minimum variance portfolio.
	Point B is the robust minimum-return efficient portfolio.
	Point C is the robust maximum-risk efficient portfolio and it is also the robust Sharpe portfolio.
	Point D is the optimal robust return-risk efficient portfolio $\bfw_\rmv$ and is computed via the formula 
	(\ref{Optimal_solution}).
	}
	\label{Fig-largeTarget}
\end{figure}

\subsection{The case: $\mu_\target < \mu_\mvp$} \label{subEF2}

This case could happen in the following scenario. 
Suppose we have a risk-free asset whose return is $r_0$ and it is usually less than $\mu_\mvp$.
We would like the portfolio return based on each $\bfbeta$ from the uncertainty set is not less than $r_0$.
That is, we set $\mu_\target = r_0$.
In this case, the well-known  tangent portfolio $\bfw_T$ in the Markowitz MV model exists
and is given by
\[
  \bfw_T = \frac{1}{b - r_0 a} \hatC^{-1} (\hatbeta - r_0 \bfone_n),
\]
and its return and variance are respectively given by \cite[Theorem~5.2.1]{brugiere2020quantitative}:
\be \label{Eq-wT}
 \bal
  \mu_T = r_0 + \frac{1}{b - r_0 a} \| \hatbeta - r_0 \bfone_n \|^2_{\hatC^{-1}} \quad \mbox{and} \quad \sigma_T^2 = \frac{1}{(b - r_0 a)^2} \| \hatbeta - r_0 \bfone_n \|^2_{\hatC^{-1}},
  \eal
\ee
where
$
a := \bfone_n^\top \hatC^{-1} \bfone_n
$ and
$
b := \bfone_n^\top \hatC^{-1} \hatbeta.
$
The robust return constraint alone will select $\bfw_T$ as it yields the largest Sharpe ratio with respect to $r_0$. 
However, when the robust risk constraint is also considered, whether
 JERO-MV model (\ref{JERO-MV}) selects $\bfw_T$ or not depends on the magnitude of $\sigma_\target$.
Two risk levels play a role in the decision. 
 Define
\bes
 \bal
  \sigma_\ell^2 := \sigma_\mvp^2 + (\mu_\mvp - \mu_\target)^2 \quad \mbox{and} \quad \sigma^2_u := \sigma^2_T + (\mu_T - \mu_\target)^2.
  \eal
\ees

\begin{theorem} \label{Thm-EF-B}
Suppose $\mu_\target < \mu_\mvp$ and $\sigma_\target \ge \sigma_\mvp$. The optimal solution of JERO-MV model (\ref{JERO-MV}) is given as follows:
\[
  \bfw_\rmv = \left\{
  \begin{array}{ll}
  		\bfw_\mvp \ & \mbox{if} \ \sigma_\target \le \sigma_\ell 
  	\\
  	\mbox{given by (\ref{Optimal_solution}) } \ & \mbox{if} \ \sigma_\target \in (\sigma_\ell, \; \sigma_u) \\
  \bfw_T \ & \mbox{if} \ \sigma_\target \ge \sigma_u .  	
  \end{array} 
  \right .
\]
\end{theorem}

\begin{figure}[h!]
	\centering
	\includegraphics[width= 0.8\textwidth, height = 0.5\textwidth]{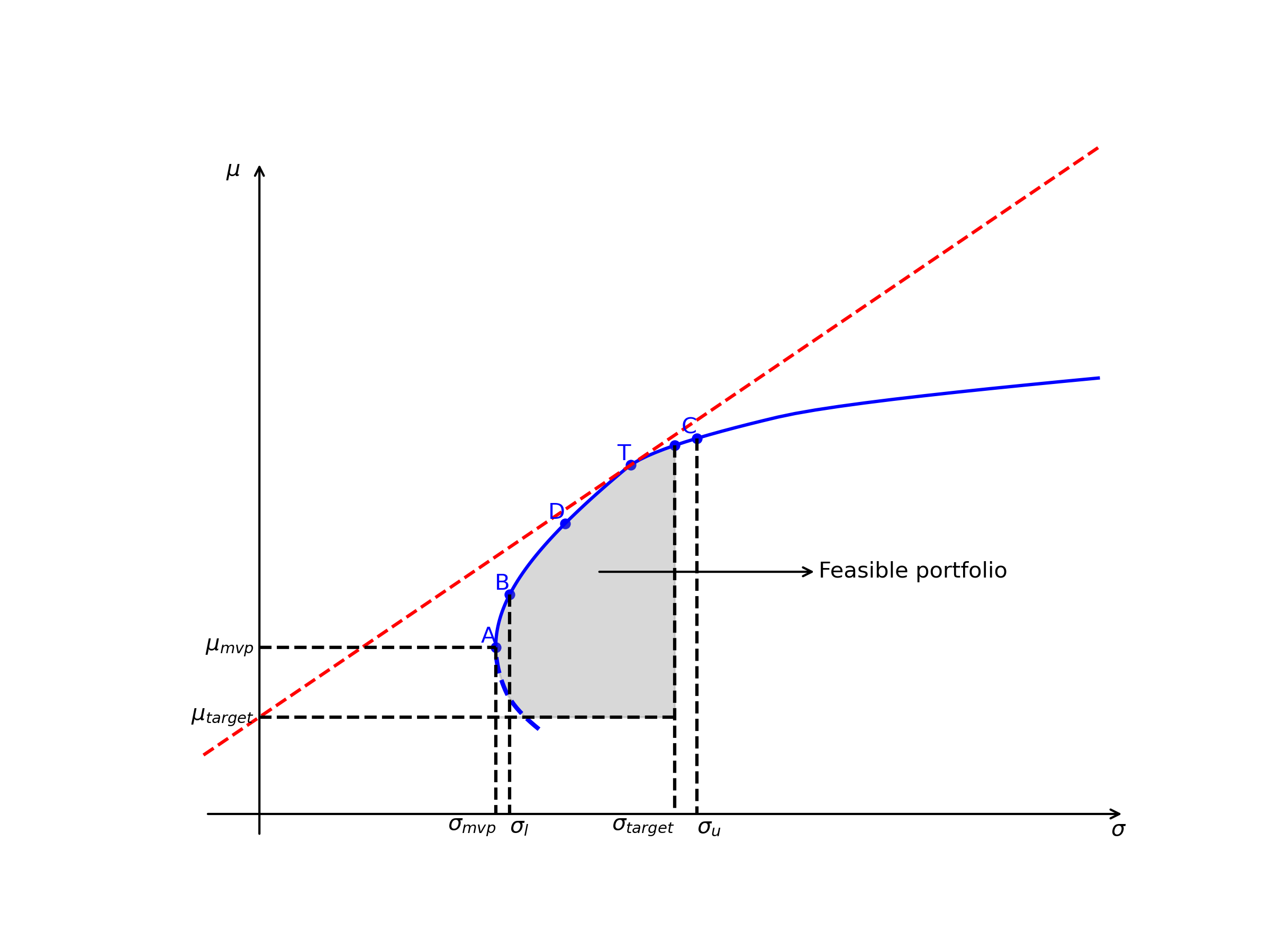}
	\caption{The case: $\mu_\target < \mu_\mvp$.
		Point A is the global minimum variance portfolio.
		Point B is the efficient portfolio corresponding to $\sigma_\ell$.
		Point C is the efficient portfolio corresponding to $\sigma_u$.
		Point T is the tangent portfolio.
		When $\sigma_\target \le \sigma_\ell$, Point A is the optimal robust return-risk portfolio.
		When $\sigma_\target \ge \sigma_u$, Point T is the optimal robust return-risk portfolio.
		In all other cases, 
		Point D is the optimal robust return-risk efficient portfolio $\bfw_\rmv$ and is computed via the formula 
		(\ref{Optimal_solution}).
	}
	\label{Fig-smallTarget}
\end{figure}

\begin{remark}
	When setting $\mu_\target = r_0$ (risk-free return), Fig.~\ref{Fig-smallTarget} is similar to the efficient frontier for the classical MV model with one key difference. There is risk-tolerance zone $(\sigma_l,\sigma_u)$. When the risk level is set above $\sigma_u^2$, JERO-MV completely ignores it and always chooses the tangent portfolio $\bfw_T$. Therefore, JERO-MV does not encourage too aggressive investments. On the other hand, if the risk level is set below $\sigma_l^2$, JERO-MV simply picks the $\bfw_\mvp$ portfolio. This means that if the investment is not risky enough, JERO would prefer the portfolio that has the lowest risk. When the risk level falls within the tolerance zone, then an efficient portfolio is selected to meet the preset return and risk targets. When $\mu_\target \geq \mu_\mvp$, JERO-MV selects an efficient portfolio just as the way it was when the risk level falls within the tolerance zone. Those comments are well reflected in both Fig.~\ref{Fig-largeTarget} and Fig.~\ref{Fig-smallTarget}. 
\end{remark}

\subsection{Including more constraints}

The analytical results reported above are valid for the budge constraint only.
In practice, there might be other constraints enforced. For example,
shorting ($w_i <0$ for some assets) is often not allowed.
We let $\bfw \in \P$ to denote a collection of such constraints. It may include 
\[
	L \le \bfw \le U \quad (\mbox{lower and upper bounds on \bfw})
\]
\[
	\| \bfw\|_1 = |w_1| + \cdots + |w_n| \leq c \quad (\ell_1\mbox{-norm constraint}),
\]
where $L, U \in \Re^n$ such that $L \le U$, and $c>0$ is a constant. 
In particular, when $L=0$ and $U = \bfone_n$, the lower and the upper bound constraint becomes no shorting constraint
$(\bfw \ge 0$).
The use of $\ell_1$-norm constraint is getting popular as it yields sparse portfolios, see \cite{Brodie2009Sparse, Fan2012Vast, Zhao2021optimal}.
We assume $\P$ be convex.
The robust portfolio selection problem (JERO-P) is then defined by 
\be \label{JERO-P}
\begin{array}{ll}
	\max\limits_{r, \bfw} & r \\
	\mbox{s.t.} & \sqrt{ e^{2r} - 1}  \sigma_{\bfw}  \le  \hatbeta^\top \bfw - \mu_\target \\ [1ex]
	&  e^{2r} \sigma_{\bfw}^2  \le  \sigma_\target^2 \\ [1ex]
	& \bfone_n^\top \bfw = 1, \ \ r \ge 0 \\ [1ex]
	& \bfw \in \P .
\end{array} 
\ee  
We need a numerical method.

The method largely follows the one suggested by \cite{Zhu2022joint} with one modification. 
It is noted that when $r$ is fixed, JERO-P model (\ref{JERO-P}) is a convex problem. 
However, one potential issue is that JERO-P model (\ref{JERO-P}) may not be feasible. 
That is, with the constraint $\bfw \in \P$, the two robust constraints may not be satisfied simultaneously.
We make the following modification. For a fixed $r$, let
\[
   \F_r := \left\{
    \bfw \ \Big| \ e^{r} \sigma_{\bfw}  \le  \sigma_\target , \ \ \bfone_n^\top \bfw = 1, \ \ \bfw \in \P
   \right\}
\]
and
\bes
  f_r(\bfw) := \left\{
  \begin{array}{ll}
       \sqrt{ e^{2r} - 1}  \sigma_{\bfw}  -  (\hatbeta^\top \bfw - \mu_\target) , \ &  \mbox{if} \ \bfw \in \F_r \\
       \infty  , \ & \mbox{otherwise}.
  \end{array}
  \right . 
\ees
In this setting, we treat the robust mean constraint as an objective over the feasible region of JERO-P model (\ref{JERO-P}) excluding the robust return constraint. 
Consequently, we solve the following JERO-Pr problem:
\be \label{JERO-Pr}
	\min_{\bfw} f_r(\bfw) \quad \mbox{s.t.} \quad \bfw \in \F_r.  
\ee
Let $f_r^*$ be the optimal objective value of JERO-Pr model (\ref{JERO-Pr}). It is easy to see that if $f_r^* >0$, then the robust return constraint is never satisfied whenever $\bfw \in \F_r$.
This implies that the fixed $r$ is too big. We need to reduce it and solve the new problem with the new $r$. 
This can be done by a bisection method on $r$.

Let us briefly explain why JERO-Pr model (\ref{JERO-Pr}) is convex and its final form.
As is a well-known technique \citep{ben2009robust}, the norm constraint can be reformulated as a second-order cone constraint by
\[
 \| \hatC^{1/2} \bfw \| \le t 
 \quad \Longleftrightarrow \quad
 \left\{
 \begin{array}{l}
 	\| \bfz \| \le t \\ [1ex]
 	\hatC^{1/2} \bfw  = \bfz .
 \end{array} 
 \right .
\]
For any given $(t, \bfw)$ define
\bes
  f_r(t, \bfw) := \left\{
  \begin{array}{ll}
       \sqrt{ e^{2r} - 1}  t  -  (\hatbeta^\top \bfw - \mu_\target) , \ & \mbox{if} \ \bfw \in \F_r \\
       \infty  , \ & \mbox{otherwise}.
  \end{array}
  \right . 
\ees
Then JERO-Pr model (\ref{JERO-Pr}) is equivalent to the following JERO-SOCP:
\be \label{JERO-SOCP}
\begin{array}{ll}
	\min\limits_{t, \bfw, \bfz} & f_r(t, \bfw) \\
	\mbox{s.t.} 
	&  e^{r} t  \le  \sigma_\target, \\[1ex]
    &  \| \bfz \| \le t, \ \hatC^{1/2} \bfw  = \bfz, \\[1ex]
	& \bfone_n^\top \bfw = 1, \ \ \bfw \in \P.
\end{array} 
\ee  
This is a second-order cone programming (SOCP) problem for fixed $r$ and hence it is convex. The bisection algorithm is stated as follows.
\begin{algorithm}[!h]
\caption{Bisection (BS) algorithm}\label{<alg-BS>}
\begin{algorithmic}[1]
  \item[1.] Initialization: ${\bf w}_0 \in \mathbb{R}^n$, $r_1 = 0$, $r_2 \in \mathbb{R}_+$ and the tolerance level $\delta$ (e.g., $10^{-5}$). Set $k:=0$.
  \item[2.] If $r_2 - r_1 \leq \delta$, stop; otherwise, go to step 3.
  \item[3.] Let $r:= (r_1 + r_2)/2$. Solve JERO-SOCP model (\ref{JERO-SOCP}) and obtain the optimal solution ${\bf w}  $.
  \item[4.] If $f_r(t,\bfw) \leq 0$, update $r_1:= r$; otherwise, update $r_2:= r$.
  \item[5.] Go to step 2.
\end{algorithmic}
\end{algorithm}

\section{Numerical Results} \label{Section-NE}

This part will study the numerical performance of JERO-MV model (\ref{JERO-MV}) and its variant when more constraints such as long-only constraint are added, in comparison with some existing robust models. 

\subsection{Location of JERO-MV portfolios}\label{subNE1}
 {\bf (a) Data sets.} We choose two real data sets for our numerical illustration. The first data set consists of 447 stocks that have appeared in the Standard \& Poor's 500 Index (S\&P500). S\&P500 is considered the most representative U.S. stock market index. Its changes draw the attention of global investors and profoundly impact global stock markets and investment strategies. The second data set is from \cite{maillard2010the}. It shows a large heterogeneity in volatilities and similarities of correlation coefficients around low levels (0\% - 10\%). In addition, this data is more suitable for constructing the numerical experiments of the case $\mu_\target < \mu_\mvp$, because its standard deviation is relatively low. We describe the two data sets below.
 
\begin{example}
    (S$\&$P500 stocks) The data set consists of 447 stocks and the data period is from January 2, 2019 to December 29, 2023. The mean and covariance matrix of the daily returns were annualized in this experiment (see \cite{qi2023geometric} for details).
\end{example}

\begin{example}
    (Agricultural assets) This data set consists of 8 light agricultural commodities taken from \cite{maillard2010the} and is given in Table \ref{Tab-agri}.
\end{example}
     
\begin{table}[!h]
\centering
\caption{Description of the agricultural assets data set}\label{Tab-agri}%
\begin{tabular}{@{}llcllllllll@{}}
\hline
Code & Return & $\sigma$ & \multicolumn{8}{c}{Correlation matrix (\%)}\\
\hline
CC  & 4.5\%   & 21.4\%  & 100 & 2.7  & 4.2   & 61.8 & 51.6 & 13.9 & 4.6  & 9.3  \\
CLC & 17.2\%  & 14.8\%  & ~   & 100  & 31.0  & 4.5  & 3.5  & 2.5  & 0.8  & 3.7  \\
CLH & 14.4\%  & 22.6\%  & ~   & ~    & 100   & 7.0  & 5.9  & 5.0  & -0.7 & 3.1  \\
CS  & 10.5\%  & 21.8\%  & ~   & ~    & ~     & 100  & 42.8 & 16.2 & 6.3  & 10.4 \\
CW  & 5.1\%   & 23.7\%  & ~   & ~    & ~     & ~    & 100  & 10.9 & 5.6  & 7.9  \\
NCT & 3.6\%   & 23.2\%  & ~   & ~    & ~     & ~    & ~    & 100  & 3.4  & 7.3  \\
NKC & 4.2\%   & 36.5\%  & ~   & ~    & ~     & ~    & ~    & ~    & 100  & 6.6  \\
NSB & 5.0\%   & 43.8\%  & ~   & ~    & ~     & ~    & ~    & ~    & ~    & 100  \\
\hline
\end{tabular}
\end{table}

{\bf (b) Numerical results.} In this experiment, we study the relationship between JERO-MV portfolio and efficient frontier. Fig. \ref{Fig-JEROMVvsMV} shows the results on S$\&$P500 stocks and Agricultural assets. 
The black curve represents the effective frontier, consists of efficient portfolios ${\bf w}_{\ef}$ with standard deviation $\sigma_{\ef}$ and expected return $\mu_{\ef}$. MVP is the global minimum variance portfolio. 
For S$\&$P500 data set, set $\mu_\target = c_1 \mu_{\ef}$ and $\sigma_\target = (1+c_2) \sigma_{\ef}$. 
Set $c_1 = 1, 1.01,1.02,\cdots,1.1$ and $c_2 = 5\%$. 
We can get the portfolio ${\bf w}_{\rmv}$ by solving JERO-MV model (\ref{JERO-MV}). 
We select five effective portfolios ${\bf w}_{\ef}$, where $\sigma_\ef = (0.0776, 0.0831, 0.0932, 0.1066, 0.1221)$ and $\mu_\ef = (0.1502, 0.2879, 0.4255, 0.5632, 0.7008)$. For each $\bfw_\ef$, there are eleven JERO-MV portfolios $\bfw_\rmv$. we circle $\bfw_\ef$ with the corresponding $\bfw_\rmv$ by a green ellipse.
For Agricultural assets, we construct the case $\mu_\target < \mu_\mvp$. 
Specifically, we set $\mu_\target = r_0 = c_1 \mu_\mvp$ and $\sigma_\target = (1+c_2) \sigma_{\ef}$. Given $c_1 = 0.7$ and $c_2 = 0.3$, we also select five effective portfolios ${\bf w}_{\ef}$, where $\sigma_\ef = (0.1062, 0.1298, 0.1675, 0.2118, 0.2594)$ and $\mu_\ef = (0.1287, 0.1651, 0.2014, 0.2378, 0.2742)$. For each $\bfw_\ef$, there is one $\bfw_\rmv$. We link $\bfw_\ef$ with its corresponding $\bfw_\rmv$.
\begin{figure}[h!]
	\centering
     \begin{subfigure}{0.45\linewidth}
		\includegraphics[width= 1\textwidth]{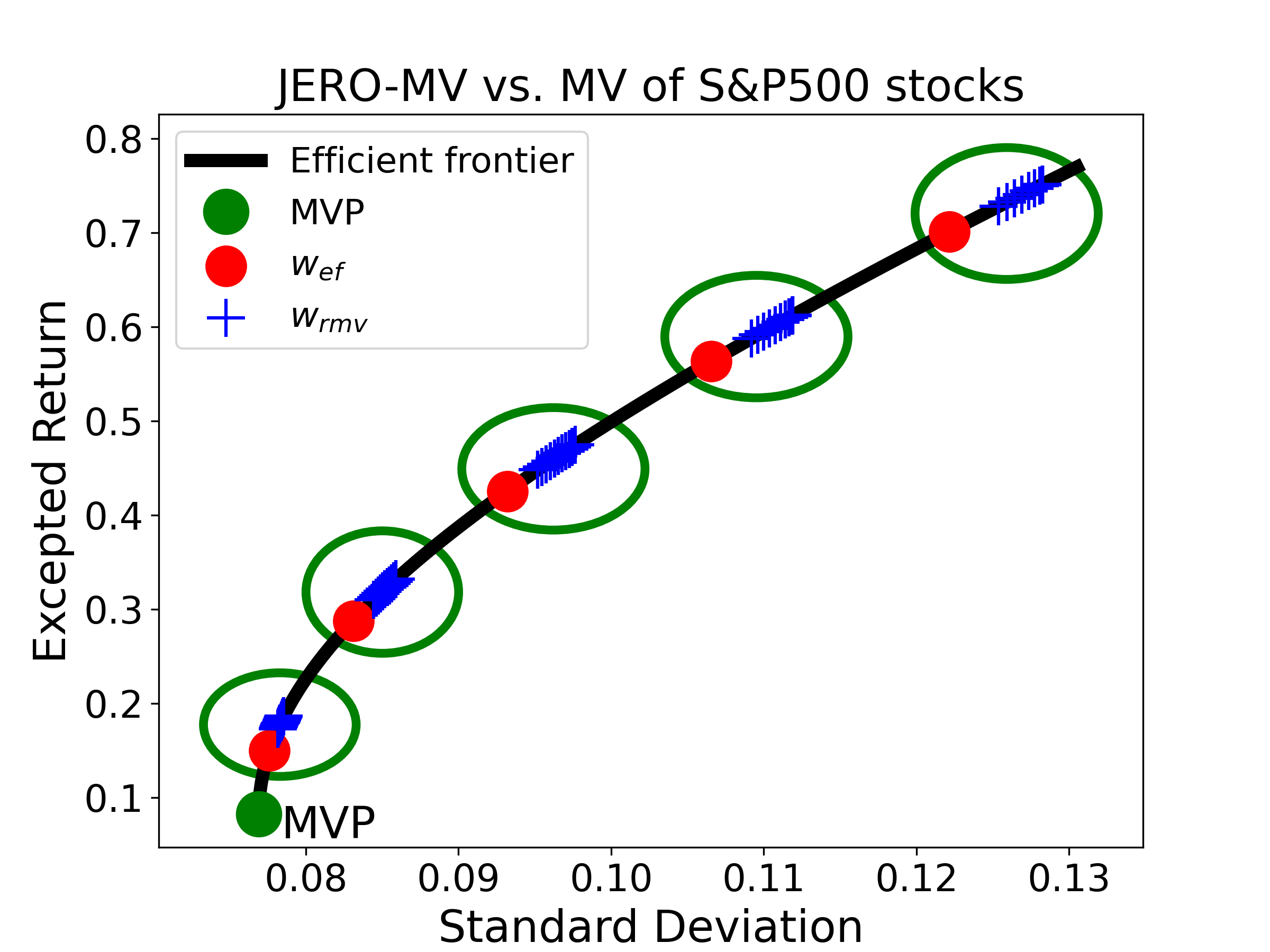}
		\subcaption{}
	 \end{subfigure}
     \begin{subfigure}{0.45\linewidth}
		\includegraphics[width= 1\textwidth]{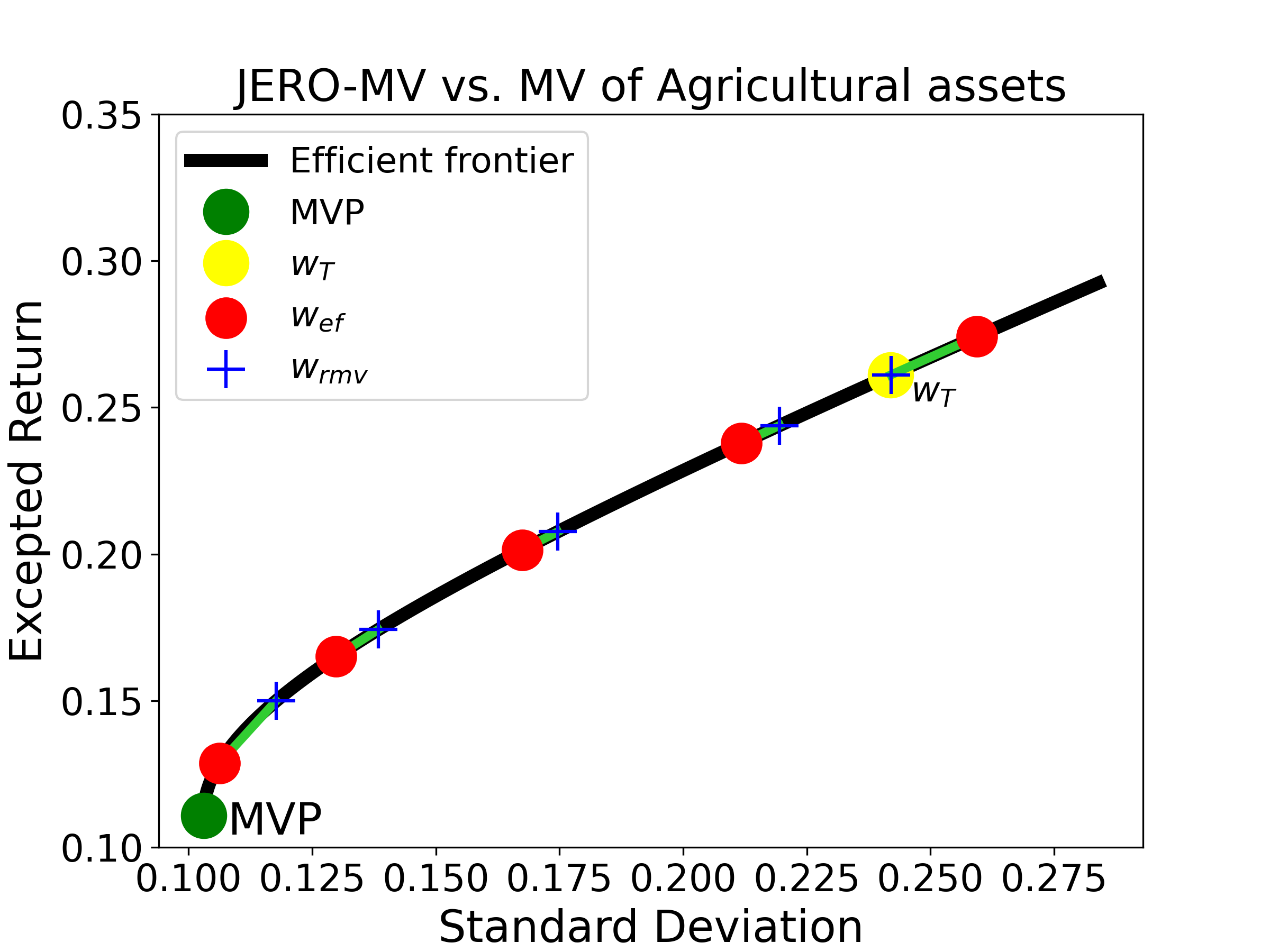}
		\subcaption{}
	 \end{subfigure}
  \caption{Comparison between JERO-MV and Efficient frontier. (a) demonstrates the results of S$\&$P500 stocks. We circle the MV portfolio $\bfw_\ef$ with the corresponding JERO-MV portfolio $\bfw_\rmv$ by a green ellipse. (b) demonstrates the results of Agricultural assets. $\bfw_T$ is defined as (\ref{Eq-wT}). We link the MV portfolio $\bfw_\ef$ with the corresponding JERO-MV portfolio $\bfw_\rmv$ by a green line.}
  \label{Fig-JEROMVvsMV}
\end{figure}

We also investigate the relationship between JERO-MV model's variant and efficient frontier. Specifically, we consider the JERO-MV model with long-only constraint (JERO-MVL). 
The design of the numerical experiment is similar to the above. 
Fig. \ref{Fig-JEROMVLvsMVL} shows the results on S$\&$P500 stocks and Agricultural assets. 
The black curve represents the effective frontier of MV model with long-only constraint (MVL), consists of efficient long-only portfolios ${\bf w}_{\efl}$ with standard deviation $\sigma_{\efl}$ and expected return $\mu_{\efl}$. 
Set $\mu_{\target} = c_1 \mu_{\efl}$ and $\sigma_{\target} = (1+c_2) \sigma_{\efl}$, we can get the portfolio ${\bf w}_{\rmvl}$ by solving JERO-P model (\ref{JERO-P}) with $\mathcal{P} = \mathbb{R}^n_+$. 
In this experiment, we set $c_1 = 0.98, 0.99, \cdots, 1.1$ and $c_2 = 0.05$ such that we can observe the result of $\mu_\target < \mu_\mvp$. We select five effective portfolios ${\bf w}_{\efl}$, where $\sigma_\efl = (0.0776, 0.0831, 0.0932, 0.1066, 0.1221)$ and $\mu_\efl = (0.1502, 0.2879, 0.4255, 0.5632, 0.7008)$. For each $\bfw_\efl$, there are thirteen JERO-MVL portfolios $\bfw_\rmvl$. We circle the MV portfolio $\bfw_\ef$ with the corresponding JERO-MV portfolio $\bfw_\rmv$ by a green ellipse in Fig. \ref{Fig-JEROMVLvsMVL}(a) and link the MV portfolio $\bfw_\ef$ with the corresponding JERO-MV portfolio $\bfw_\rmv$ by a green line in Fig. \ref{Fig-JEROMVLvsMVL}(b).

\begin{figure}[h!]
	\centering
     \begin{subfigure}{0.45\linewidth}
		\includegraphics[width= 1\textwidth]{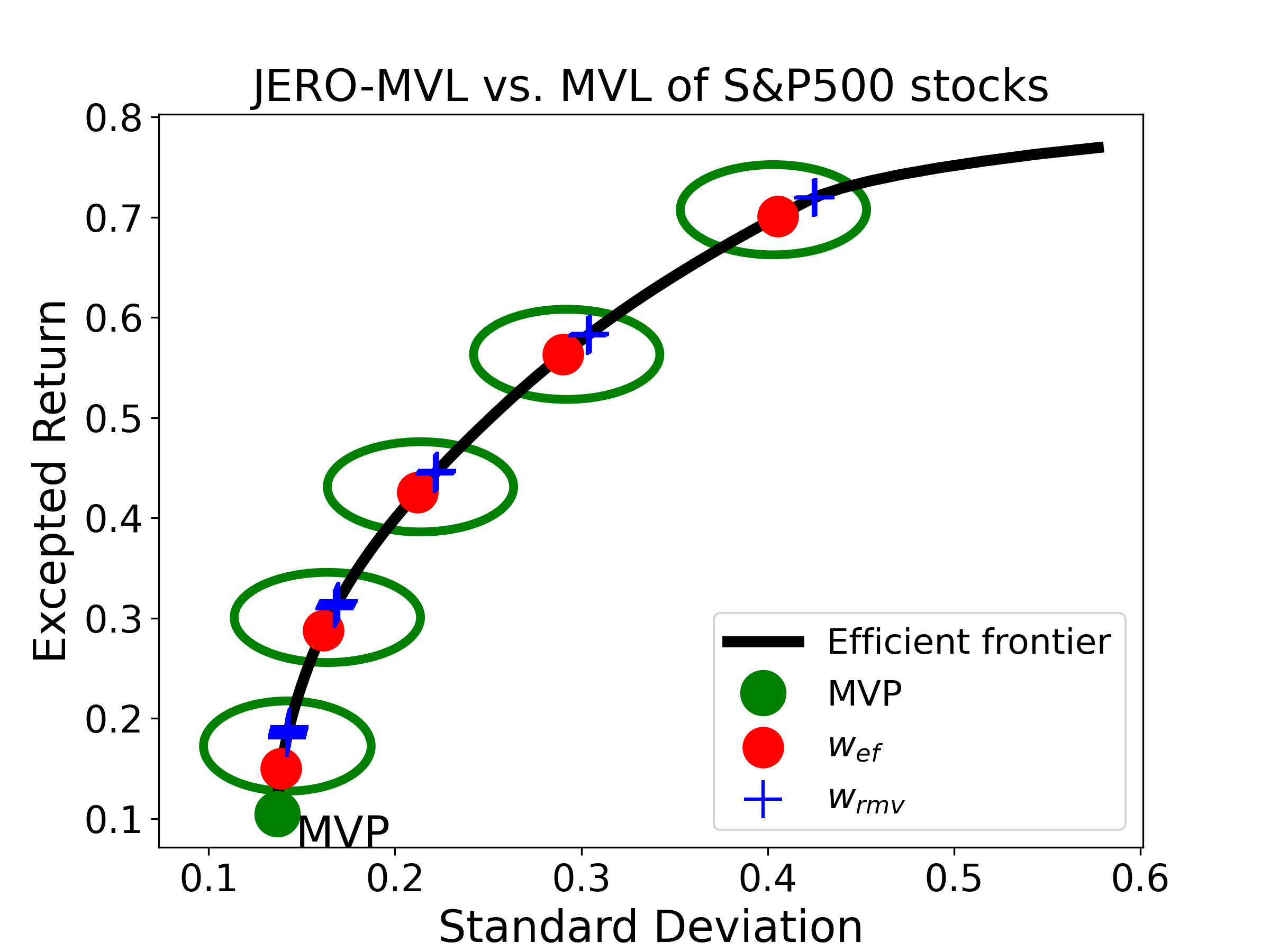}
		\subcaption{}
	 \end{subfigure}
     \begin{subfigure}{0.45\linewidth}
		\includegraphics[width= 1\textwidth]{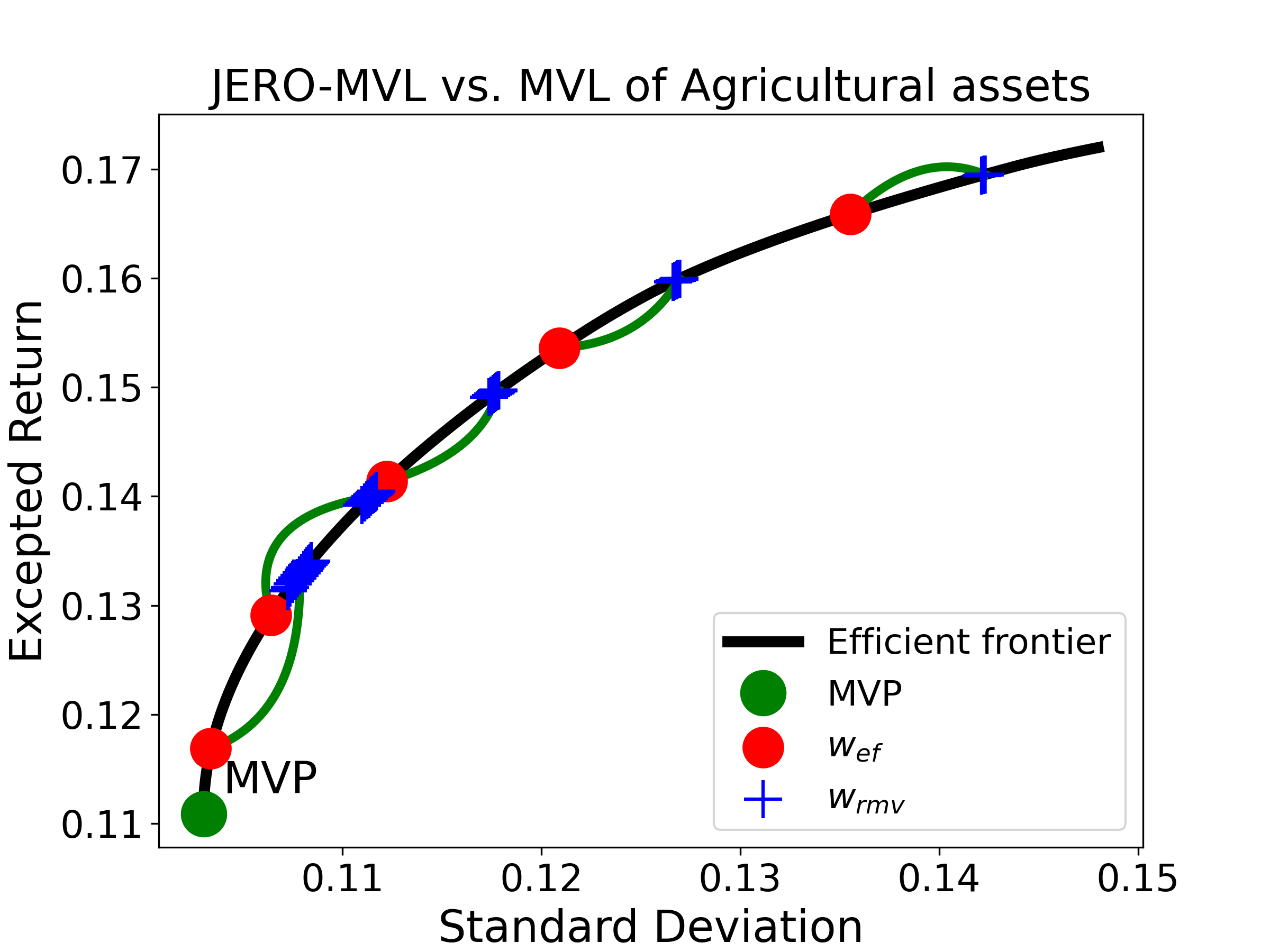}
		\subcaption{}
	 \end{subfigure}
  \caption{Comparison between JERO-MVL and Efficient frontier. (a) demonstrates the results of S$\&$P500 stocks. We circle the MVL portfolio $\bfw_\efl$ with the corresponding JERO-MVL portfolio $\bfw_\rmvl$ by a green ellipse. (b) demonstrates the results of Agricultural assets. The green curves are used to link $\bfw_\efl$ with the corresponding $\bfw_\rmvl$.}
  \label{Fig-JEROMVLvsMVL}
\end{figure}

\vspace{\baselineskip}
{\bf (c)  Key observations.} We have the following observations.

\begin{itemize}
\item[1.] {\bf $\bfw_\rmv$ is in a neighborhood of the corresponding $\bfw_\ef$ and has the higher return and Sharpe ratio.} 
In Fig. \ref{Fig-JEROMVvsMV}, all $\bfw_\rmv$ are on the efficient frontier curves.
From Fig. \ref{Fig-JEROMVvsMV}(a), we found that all $\bfw_\rmv$ are in a small neighborhood above the corresponding $\bfw_\ef$ along the efficient frontier curve. We report the returns, risks and Sharpe ratios with $c_1 = 1$ and $c_2 = 0.05$ in Table \ref{Tab-perform-MV}. For $c_1 = 1.01,\cdots,1.1$, the return and Sharpe ratio of $\bfw_\rmv$ will be larger. In the five tests, the mean return of $\bfw_\rmv$ is $2.38\%$ higher than that of $\bfw_\ef$. Furthermore, the Sharpe ratios of $\bfw_\rmv$ are larger than that of the corresponding $\bfw_\ef$. 
In Fig. \ref{Fig-JEROMVvsMV}(b), we study the case $\mu_\target < \mu_\mvp$. When $\sigma_\target \in (\sigma_l, \sigma_u) = (0.1083, 0.3037)$, all $\bfw_\rmv$ are above the corresponding $\bfw_\ef$ along the efficient frontier curve. From Table \ref{Tab-perform-MV}, the mean return of $\bfw_\rmv$ is $0.59\%$ higher than that of $\bfw_\ef$. When $\sigma_\target > \sigma_u$, $\bfw_\rmv$ is $\bfw_T$, defined as (\ref{Eq-wT}). In both cases, JERO-MV portfolio has a larger Sharpe ratio.
For investors with a relatively low return target (i.e. $\mu_\target < \mu_\mvp$), they can get a portfolio with the higher return and Sharpe ratio by appropriately increasing the risk target.

\item[2.] {\bf $\bfw_\rmvl$ is efficient and in a neighborhood above the corresponding $\bfw_\efl$ along the efficient frontier curve.} 
In Fig. \ref{Fig-JEROMVLvsMVL}, all $\bfw_\rmvl$ are on the efficient frontier curves. 
This is related to our bisection algorithm.
In fact, maximizing $r$ such that $f_r(t,\bfw)\leq 0$, is equivalent to maximizing the Sharpe ratio over the feasible region $\F_r$. 
Denote the optimal $r$ as $r_\ast$. 
When $r=r_\ast$, we obtain a maximum Sharpe ratio portfolio $\bfw_\ast$ among all $\bfw \in \F_{r}$, which is therefore an efficient portfolio.
From Fig. \ref{Fig-JEROMVLvsMVL}(a), we found that all $\bfw_\rmvl$ are above the corresponding $\bfw_\efl$ along the curve, even when we set $\mu_\target < \mu_\efl$. We report the returns, risks and Sharpe ratios with $c_1 = 1$ and $c_2 = 0.05$ in Table \ref{Tab-perform-MVL}. The mean return of $\bfw_\rmvl$ is $2.30\%$ higher than that of $\bfw_\efl$.   
In Fig. \ref{Fig-JEROMVLvsMVL}(b), the performance of $\bfw_\rmvl$ is similar to that of S$\&$P500 stocks. 
From Table \ref{Tab-perform-MVL}, the mean difference of returns is $0.87\%$. 
With the increase of $\mu_\target$ and $\sigma_\target$, Sharpe ratio first increases and then decreases. Therefore, investors should set reasonable return and risk targets according to their investment style. 

\end{itemize}

\begin{table}[!h]
\centering
\caption{Returns, risks and Sharpe ratios of JERO-MV and MV portfolios}\label{Tab-perform-MV}%
\begin{tabular}{@{}lll@{}}
\hline
~ & ~~~~~~~~~~~~$\bfw_\ef$ & ~~~~~~~~~~~~$\bfw_\rmv$ \\
\hline
\multicolumn{3}{l}{\textbf{S$\&$P500 stocks}} \\
\hline
Return       & $(0.1502, 0.2879, 0.4255, 0.5632, 0.7008)$ & $(0.1734, 0.3099, 0.4481, 0.5875, 0.7277)$ \\
Risk         & $(0.0776, 0.0831, 0.0932, 0.1066, 0.1221)$ & $(0.0782, 0.0845, 0.0952, 0.1092, 0.1254)$  \\
Sharpe ratio & $(1.5490, 3.1035, 4.2436, 5.0019, 5.4939)$ & $(1.8338, 3.3124, 4.3918, 5.1053, 5.5638)$ \\
\hline
\multicolumn{3}{l}{\textbf{Agricultural assets}} \\
\hline
Return       & $(0.1287, 0.1651, 0.2014, 0.2378, 0.2742)$ & $(0.1499, 0.1742, 0.2076, 0.2438, 0.2611)$ \\
Risk         & $(0.1062, 0.1298, 0.1675, 0.2118, 0.2594)$ & $(0.1177, 0.1384, 0.1746, 0.2195, 0.2420)$  \\
Sharpe ratio & $(0.4811, 0.6741, 0.7391, 0.7564, 0.7579)$ & $(0.6142, 0.6980, 0.7446, 0.7572, 0.7582)$ \\
\hline
\end{tabular}
\end{table}

\begin{table}[!h]
\centering
\caption{Returns, risks and Sharpe ratios of JERO-MVL and MVL portfolios}\label{Tab-perform-MVL}%
\begin{tabular}{@{}lll@{}}
\hline
 ~ & ~~~~~~~~~~~~$\bfw_\efl$ & ~~~~~~~~~~~~$\bfw_\rmvl$ \\
\hline
\multicolumn{3}{l}{\textbf{S$\&$P500 stocks}} \\
\hline
Return       & $(0.1502, 0.2879, 0.4255, 0.5632, 0.7008)$ & $(0.1826, 0.3116, 0.4453, 0.5831, 0.7199)$ \\
Risk         & $(0.1389, 0.1616, 0.2121, 0.2902, 0.4054)$ & $(0.1422, 0.1680, 0.2218, 0.3040, 0.4252)$ \\
Sharpe ratio & $(0.8654, 1.5960, 1.8647, 1.8374, 1.6547)$ & $(1.0731, 1.6762, 1.8724, 1.8194, 1.6225)$ \\
\hline
\multicolumn{3}{l}{\textbf{Agricultural assets}} \\
\hline
Return       & $(0.1169, 0.1291, 0.1414, 0.1536, 0.1659)$ & $(0.1320, 0.1396, 0.1494, 0.1598, 0.1695)$ \\
Risk         & $(0.1034, 0.1064, 0.1122, 0.1209, 0.1355)$ & $(0.1075, 0.1112, 0.1176, 0.1268, 0.1423)$ \\
Sharpe ratio & $(0.3859, 0.4897, 0.5739, 0.6336, 0.6560)$ & $(0.5116, 0.5629, 0.6156, 0.6530, 0.6500)$ \\
\hline
\end{tabular}
\end{table}

\subsection{Comparison with some classical strategies}\label{subNE2}

{\bf (A) Data sets.} The main purpose of this subsection is to evaluate the out-of-sample performance of JERO-MV model (\ref{JERO-MV}). 
Out-of-sample performance, which has been used in almost all portfolio literature, can reflect the performance level of the model in practical application. 
The data sets above are not enough for this purpose. 
Hence, we choose more data sets of different stock markets and include stocks of large, medium and small companies. Specifically, we consider DAX Index (DAX), Hangseng Index (HS), NASDAQ 100 Index (NDX100), Standard \& Poor's 500 Index (S\&P500), Russell 1000 Index (Russell1000), and Russell 2000 Index (Russell2000). 
We use the weekly return rate, which can not only capture the mid-term market dynamics, but also avoid excessive noise. 
The data period from January 7, 2019 to December 31, 2023. 
The properties of market indices are given in Table \ref{Tab-dataset}: ID, name (Dataset), the number of assets (N), the data period (Period), data sources (Source), and frequency (Frequency). 

\begin{table}[!h]
\centering
\caption{Description of the data sets}\label{Tab-dataset}%
\begin{tabular}{@{}llllll@{}}
\hline
ID & Dataset & $N$ & Period & Source & Frequency\\
\hline
   1 & DAX          & 27   & 07/01/2019-31/12/2023 & Yahoo finance & Weekly\\
   2 & HS           & 67   & 07/01/2019-31/12/2023 & Yahoo finance & Weekly\\
   3 & NDX100       & 90   & 07/01/2019-31/12/2023 & Yahoo finance & Weekly\\
   4 & S\&P500      & 485  & 07/01/2019-31/12/2023 & Yahoo finance & Weekly\\
   5 & Russell1000  & 859  & 07/01/2019-31/12/2023 & Yahoo finance & Weekly\\
   6 & Russell2000  & 1214 & 07/01/2019-31/12/2023 & Yahoo finance & Weekly\\
\hline
\end{tabular}
\end{table}

{\bf (B) Evaluation methodology.} We use the rolling window method to design the experiments. We collect time series data sets of weekly return rates for $T$ weeks. 
Denote the length of the estimation window as $M < T$.
Starting at $t = M + 1$, we use the data in the previous $M$ weeks to estimate the optimal weights for the $t$-th week. 
These weights are then used to compute the return rate in the $t$-th week. We continue doing this until $t = T$. 
At the end of this progress, we obtain $Q = T - M$ out-of-sample returns, $r_1,\cdots,r_q$, for each strategy and data set considered.

The following measures are used to evaluate the out-of-sample performances.

1. The out-of-sample mean returns (MR), which is defined as follows:
\begin{equation}
 \begin{aligned}
  \text{MR} = \frac{1}{Q} \sum\limits_{q=1}^Qr_q.
 \end{aligned}
\end{equation}

2. The out-of-sample standard deviation (SD) of the portfolio:
\begin{equation}
 \begin{aligned}
  \text{SD} = \sqrt{\frac{1}{Q-1}\sum_{q=1}^Q(r_q-MR)^2},
 \end{aligned}
\end{equation}
which is used to measure the risk of portfolio.

3. The out-of-sample Sharpe ratio (SR) is the sample mean of out-of-sample excess returns divided by their standard deviation:
\begin{equation}
 \begin{aligned}
  \text{SR}=\frac{\text{MR} - r_f}{\text{SD}},
 \end{aligned}
\end{equation}
where $r_f$ is a risk-free rate.

4. The stability index (SI) is the average sum of the absolute value of the trades across the different assets:
\begin{equation}
 \begin{aligned}
  \text{SI} = \frac{1}{n(Q-1)} \sum\limits_{q=1}^{Q-1}\sum\limits_{i=1}^{n}(\vert w_i^{q+1}-w_i^q \vert),
 \end{aligned}
\end{equation}
where $w_i^{q+1}$ is the optimal weight in the $(q+1)$-th iteration for the $i$-th asset and $w_i^q$ is the weight of the strategy on the $i$-th asset in the $q$-th iteration. The smaller SI, the better performance.

{\bf (C) Evaluation strategies.} In order to evaluate the out-of-sample performance of our model, we compare it with some commonly used models, including deterministic optimization models and robust optimization models.
In deterministic optimization models, we consider the two basic models: the mean-variance (MV) model \citep{markowitz1952, best2000the} and the global minimum variance (GMV) model \citep{Chopra1993The, jagannathan2003risk, behr2013on}. The well-known MV model has been widely accepted by researchers and practitioners \citep{Luisa2021An}. However, MV model is very sensitive to input errors. It is well known that estimating the mean is more challenging than estimating the covariance of asset returns \citep{Merton1980On}. Hence, GMV model, which rely solely on the estimation of covariance, is used to reduce the effect of input errors. 
 In robust optimization models, we consider the robust optimization model with ellipsoidal uncertainty set (RO) \citep{Xidonas2020robust, Georgantas2021robust} and JERO-M model \citep{Zhu2022joint}.
 All models in the experiments have the nonnegative constraint, which is widely used and closer to the actual operation \citep{frost1988for, jagannathan2003risk}.
 Table \ref{strategies} shows the portfolio strategies we consider in this paper.

\begin{table}[!h]
\centering
\caption{Portfolio strategies}\label{strategies}%
\begin{tabular}{@{}llll@{}}
\hline
Group & Model & Abbr. & Refer. \\
\hline
   (1) & Our model in this paper & ~  & ~\\
   ~   & ~~~~JERO with robust return and risk constraints & JERO-MV  & This paper\\
\hline
   (2) & Deterministic model & ~  & ~\\
   ~   & ~~~~Mean-variance model & MV  & \cite{Luisa2021An}\\
   ~   & ~~~~Global minimum-variance model & GMV  & \cite{jagannathan2003risk}\\
\hline
   (3) & Robust model & ~  & ~\\
   ~   & ~~~~JERO with robust return constraint & JERO-M  & \cite{Zhu2022joint}\\
   ~   & ~~~~Robust model under the ellipsoidal uncertainty set  & RO  & \cite{Georgantas2021robust}\\
\hline
\end{tabular}
\end{table}

{\bf (D) Numerical results and observations.} In this paper, we set $T = 260,~M = 156$. The mean interest rates of 10-year government bonds in Germany, China and the United States from January 2022 to December 2023 are used as the risk-free rate $r_f$ for the portfolios in the respective markets. Specifically, for DAX and HS, $r_f$ are $1.85\%$ and $3.26\%$, respectively. For the rest data sets, we set $r_f = 3.46\%$. Setting the size $\epsilon$ of the ellipsoidal uncertainty set $\mathcal{U}$ in RO equal to the $\alpha$-th percentile $\chi^2_{\alpha,N}$ of the $\chi^2$ distribution with $N$ degrees of freedom, implies that the $\alpha\%$ of the estimated expected returns will lie inside $\mathcal{U}$. This setting was recently studied in \cite{Georgantas2021robust}. In this paper, we set $\alpha = 95$.

{\bf (D1) Out-of-sample performances: JERO-MV portfolios can trade-off the return against the risk on the premise of meeting the return and risk targets.} Among the five portfolio strategies we considered, the out-of-sample performances of the JERO-MV portfolio are at a medium level. It meets both return and risk targets while having a relatively high Sharpe ratio and low transaction costs. JERO-M and MV portfolios are more likely to get high out-of-sample mean returns and Sharpe ratios, but they have to bear relatively high risk and transaction costs. RO and GMV portfolios have the lowest risk and cost, but both the mean return and Sharpe ratio are also the lowest, even lower than the risk-free interest rate. The details of observations are as following.

In Table \ref{Tab-oos}, we show the out-of-sample performances. We can observe that JERO-M portfolios have the largest mean returns on average with $20.29\%$. And for NDX100, S$\&$P500, Russell1000 and Russell2000, the mean returns of JERO-M portfolios are the largest and significantly larger than those of other strategies. MV portfolios have the second largest mean returns on average with $17.25\%$. Except for DAX, GMV portfolios have the smallest mean returns, which are much smaller than the return targets. Although the mean returns of JERO-MV portfolios are smaller than those of JERO-M and MV portfolios, they are much larger than the mean returns of the other portfolios. Furthermore, the mean returns of almost all JERO-MV and JERO-M portfolios are larger than the return targets.  

For DAX and HS, GMV and MV portfolios have the largest Sharpe ratio, respectively. For the rest data sets, JERO-M portfolios have the largest Sharpe ratio. The Sharpe ratios of JERO-M and MV portfolios are the largest on average with $1.6857$ (JERO-M) and $1.5566$ (MV). The Sharpe ratios of JERO-MV portfolios are the third largest on average, but they are the second largest for NDX100, S$\&$P500 and Russell2000. The Sharpe ratios of RO and GMV portfolios are even negative, since their mean returns sometimes are lower than $r_f$.  

The standard deviations of the various portfolios on average are $3.13\%$ (GMV), $3.13\%$ (RO), $7.70\%$ (JERO-MV), $8.00\%$ (MV), $8.45\%$ (JERO-M). For all data sets, GMV portfolios have the smallest standard deviations, which are significantly smaller than JERO-MV, JERO-M and MV portfolios. RO and GMV portfolios have the similar standard deviations. The standard deviations of JERO-MV portfolios smaller than JERO-M and MV portfolios.

As for the stability index, GMV portfolios exhibit the smallest stability index on average with $0.0090$. The averages of stability indices of JERO-MV, JERO-M and MV are much larger than RO and GMV portfolios. However, for S$\&$P500, Russell1000 and Russell2000 data sets, which have a large number of assets, the stability indices of all strategies are similar and small, which will reduce the transaction cost of the portfolios. 

\begin{table}[!h]
  \tabcolsep=0.2cm 
  \centering
  \caption{Out-of-sample performances}\label{Tab-oos}
 \begin{tabular}{cccccccc}
\hline
   ~ & ~ & DAX & HS & NDX100 & S$\&$P500 & Russell1000 & Russell2000\\[1 ex]
   ~ & ~ & N=27 & N=67 & N=90 & N=485 & N=859 & N=1214\\
\hline
  \multirow{4}*{JERO-MV} & MR & 0.0088 & 0.1674 & 0.1662 & 0.2490 & 0.2185 & 0.1030 \\
                      ~   & SR & -0.2363 & 1.6175 & 1.5888 & 2.3242 & 2.1256 & 0.9047 \\
                      ~   & SD & 0.0412 & 0.0833 & 0.0829 & 0.0923 & 0.0865 & 0.0756 \\
                      ~   & SI & 0.0105 & 0.0021 & 0.0021 & 0.0006 & 0.0003 & 0.0002 \\
\hline
  \multirow{4}*{JERO-M}   & MR & 0.0119 & 0.1933 & 0.1933 & 0.4125 & 0.2975 & 0.1090 \\
                      ~   & SR & -0.1543 & 1.6075 & 1.7662 & 3.2045 & 2.7091 & 0.9814 \\
                      ~   & SD & 0.0429 & 0.0836 & 0.0899 & 0.1179 & 0.0970 & 0.0758 \\
                      ~   & SI & 0.0127 & 0.0019 & 0.0019 & 0.0004 & 0.0002 & 0.0002 \\
\hline
  \multirow{4}*{RO}       & MR & 0.0221 & 0.0008 & 0.0327 & 0.0121 & 0.0215 & 0.0315 \\
                      ~   & SR & 0.1143 & -1.0623 & -0.0573 & -0.7027 & -0.4269 & -0.6861 \\
                      ~   & SD & 0.0317 & 0.0299 & 0.0328 & 0.0321 & 0.0306 & 0.0308 \\
                      ~   & SI & 0.0029 & 0.0012 & 0.0011 & 0.0003 & 0.0002 & 0.0001\\
\hline
  \multirow{4}*{MV}       & MR & 0.0099 & 0.2643 & 0.1781 & 0.2488 & 0.2348 & 0.0993 \\
                      ~   & SR & -0.2054 & 2.5182 & 1.6424 & 2.2931 & 2.2322 & 0.8591 \\
                      ~   & SD & 0.0419 & 0.0920 & 0.0874 & 0.0934 & 0.0897 & 0.0753 \\
                      ~   & SI & 0.0018 & 0.0020 & 0.0030 & 0.0006 & 0.0003 & 0.0002\\
\hline
  \multirow{4}*{GMV}      & MR & 0.0242 & 0.0013 & 0.0330 & 0.0124 & 0.0218 & 0.0137 \\
                      ~   & SR & 0.1801 & -1.0448 & -0.0486 & -0.6914 & -0.4165 & -0.6714 \\
                      ~   & SD & 0.0317 & 0.0299 & 0.0326 & 0.0321 & 0.0306 & 0.0308 \\
                      ~   & SI & 0.0028 & 0.0012 & 0.0010 & 0.0003 & 0.0002 & 0.0001\\
\hline
 \end{tabular}
\end{table}

{\bf (D2) When are the robust constraints are active?} 
We define the two constrait functios respectively by
\[
  g_1(r,\bfw) := \sqrt{e^{2r}-1}\| \hatC^{1/2} \bfw \| -\hatbeta^\top \bfw - \mu_\target \quad \mbox{and} \quad 
  g_2(r,\bfw) := e^r \|  \hatC^{1/2} \bfw \| - \sigma_\target
\]
For a portfolio $\bfw$, $g_i(r,\bfw)\leq 0, \ i=1,2$ means that $(r,\bfw)$ is feasible for the problem (\ref{JERO-P}). 
Let $r_\ast$ and $\bfw^\ast$ be the optimal solution of (\ref{JERO-P}). 
Then $g_1(r_\ast,\bfw^\ast)=0$ means that the robust return constraint is active at the optimal solution and the robust risk constraint is active if $g_2(r_\ast,\bfw^\ast)=0$. When both constraints are active, we have a portfolio satisfying $r_m(\bfw^\ast) = r_v(\bfw^\ast)$.
The portfolio $\bfw^\ast$ corresponds to Point D in Fig. \ref{Fig-largeTarget} and Fig. \ref{Fig-smallTarget}. 
In our extensive tests, the robust return constraint is almost always active. The robust risk constraint is sometimes not active.

\begin{figure}[h!]
	\centering
     \begin{subfigure}{0.45\linewidth}
		\includegraphics[width= 1\textwidth]{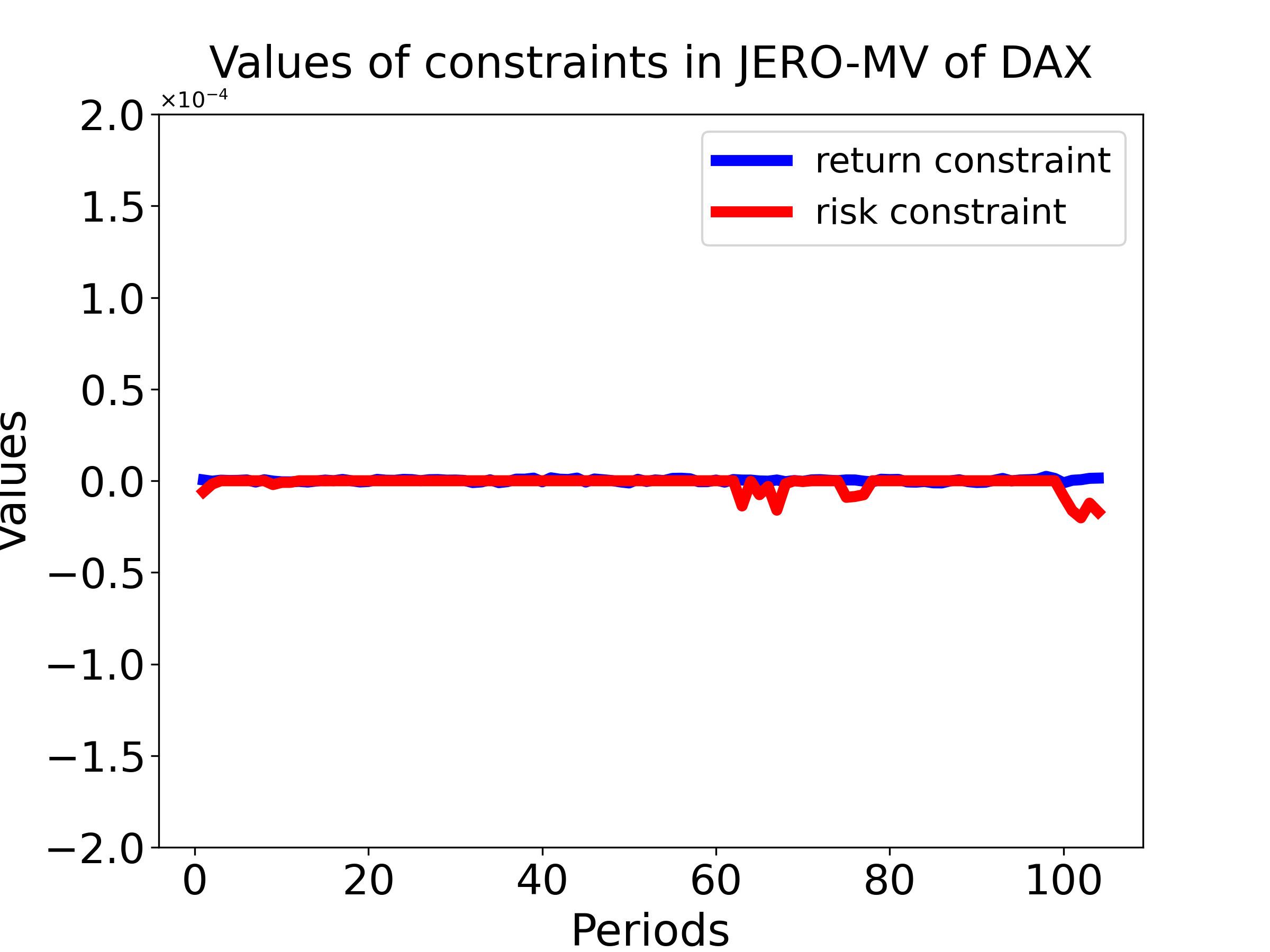}
		\subcaption{}
	 \end{subfigure}
     \begin{subfigure}{0.45\linewidth}
		\includegraphics[width= 1\textwidth]{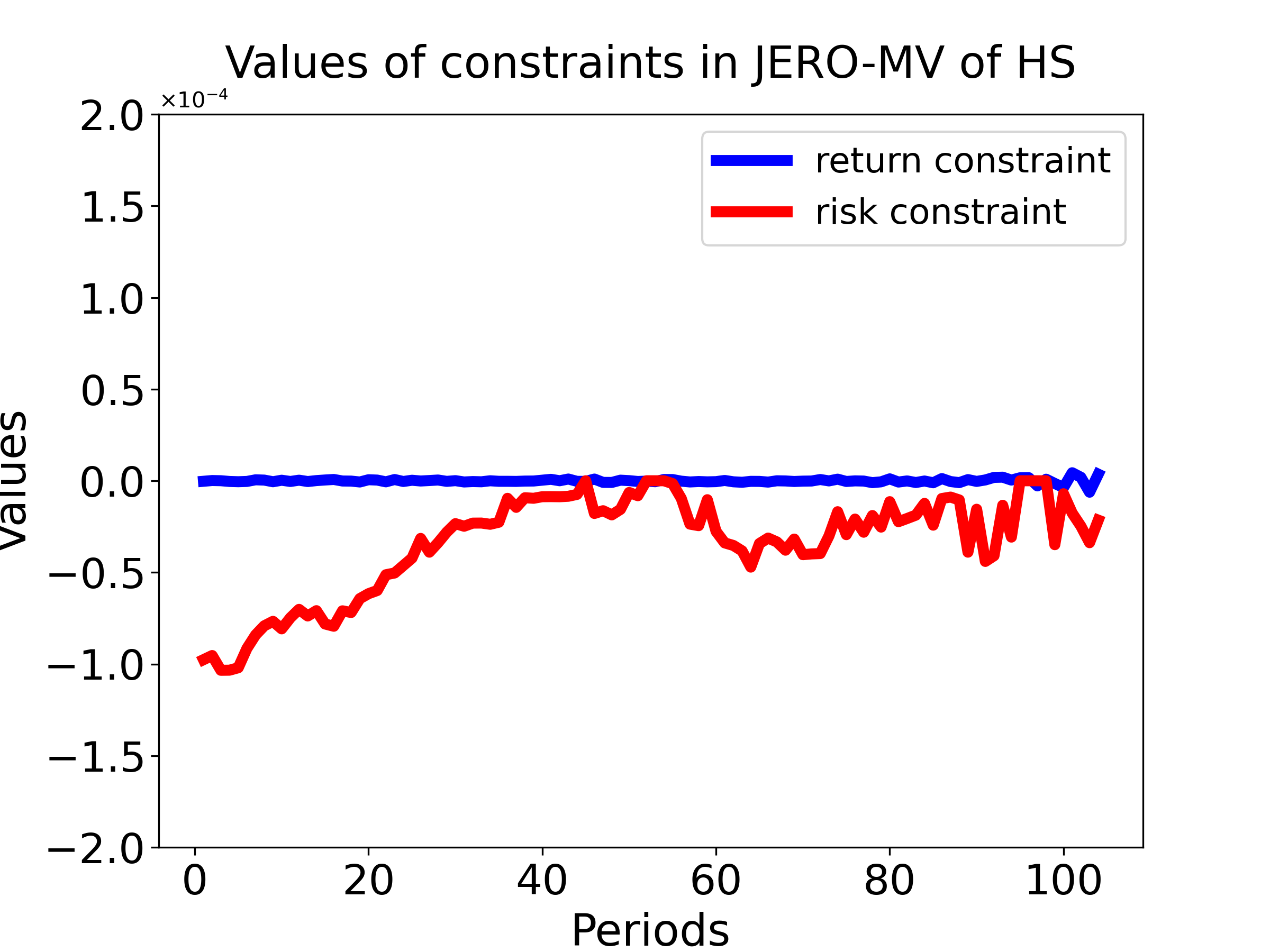}
		\subcaption{}
	 \end{subfigure}
       \begin{subfigure}{0.45\linewidth}
		\includegraphics[width= 1\textwidth]{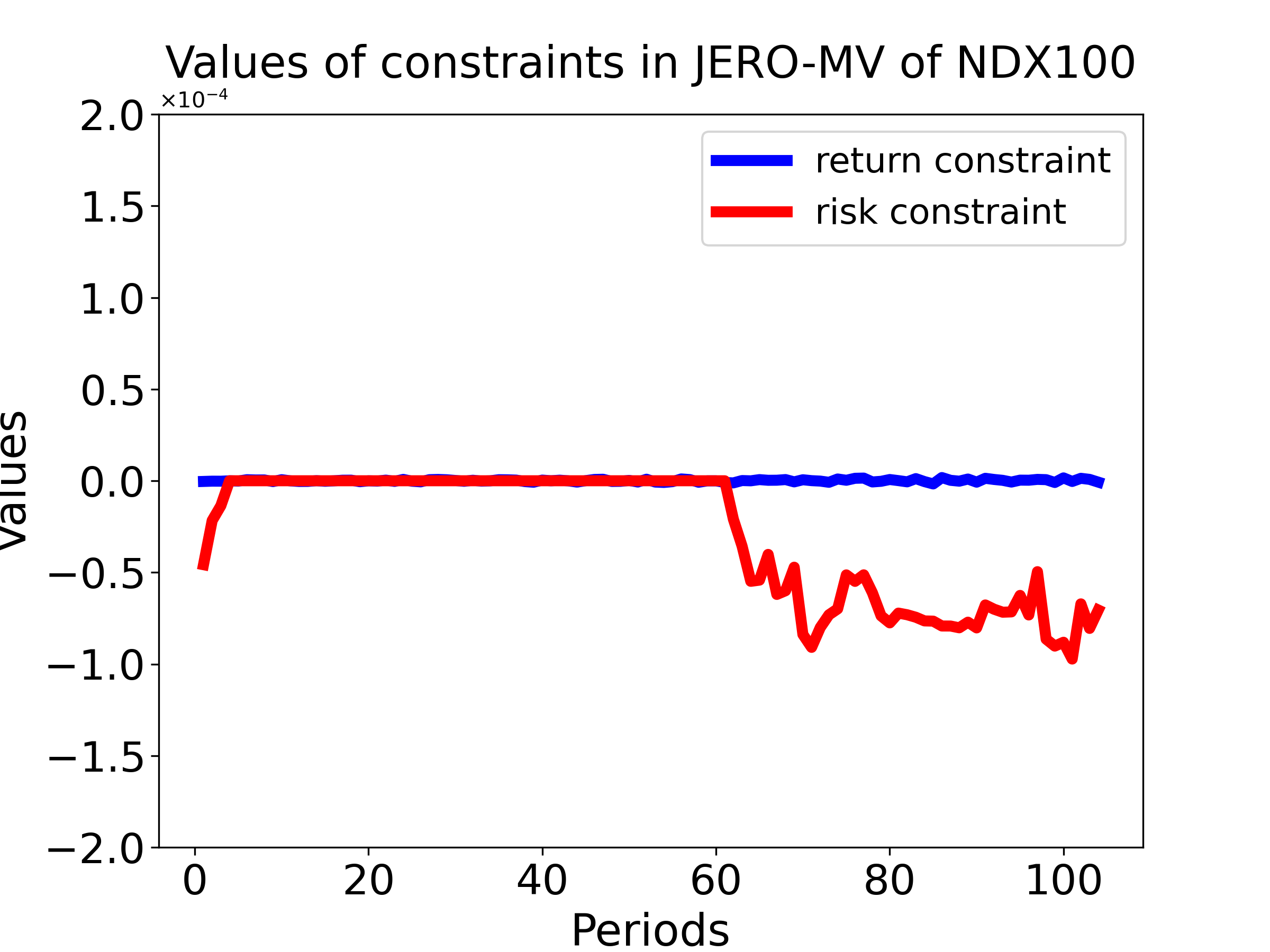}
		\subcaption{}
	 \end{subfigure}
     \begin{subfigure}{0.45\linewidth}
		\includegraphics[width= 1\textwidth]{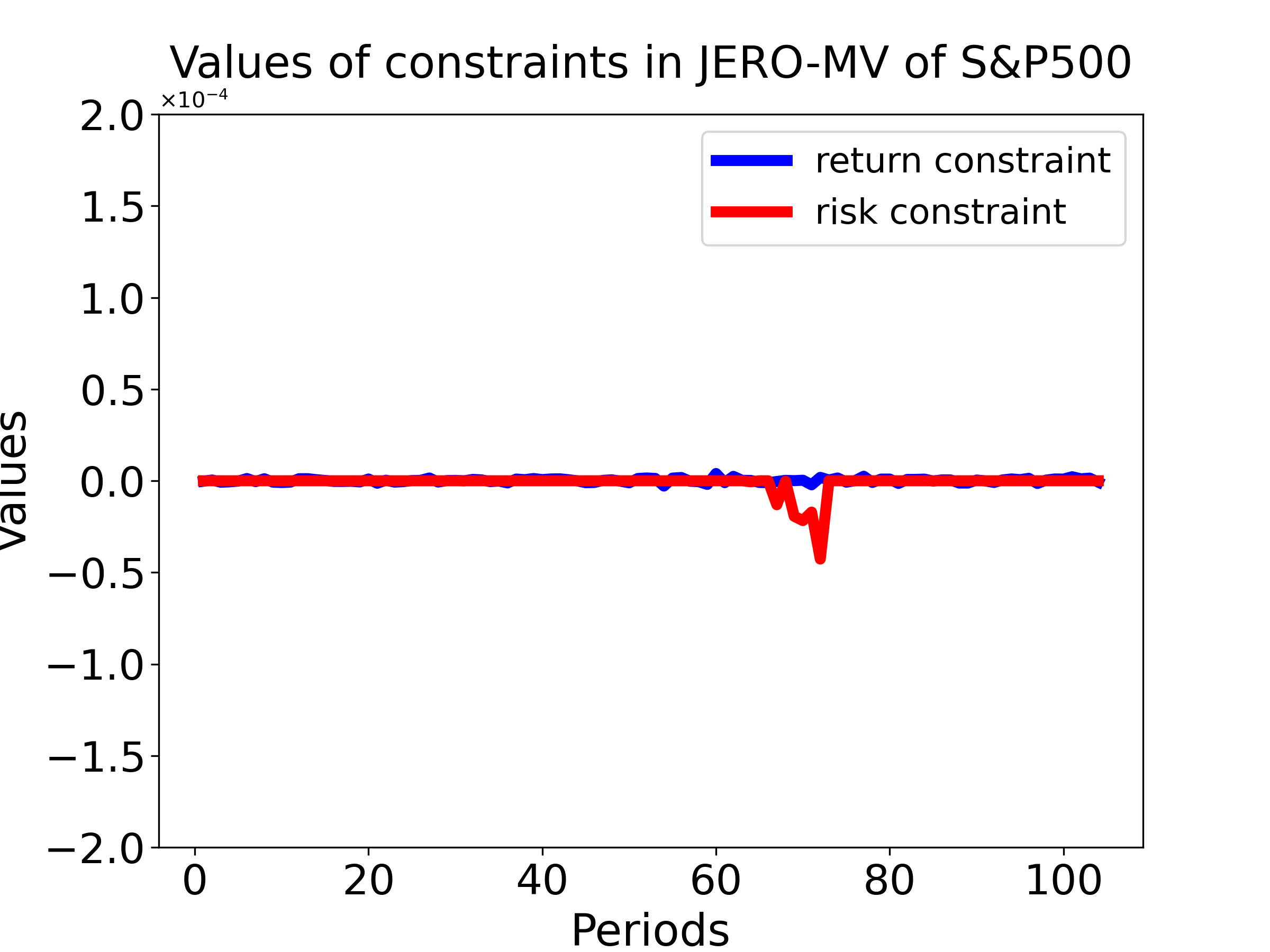}
		\subcaption{}
	 \end{subfigure}
       \begin{subfigure}{0.45\linewidth}
		\includegraphics[width= 1\textwidth]{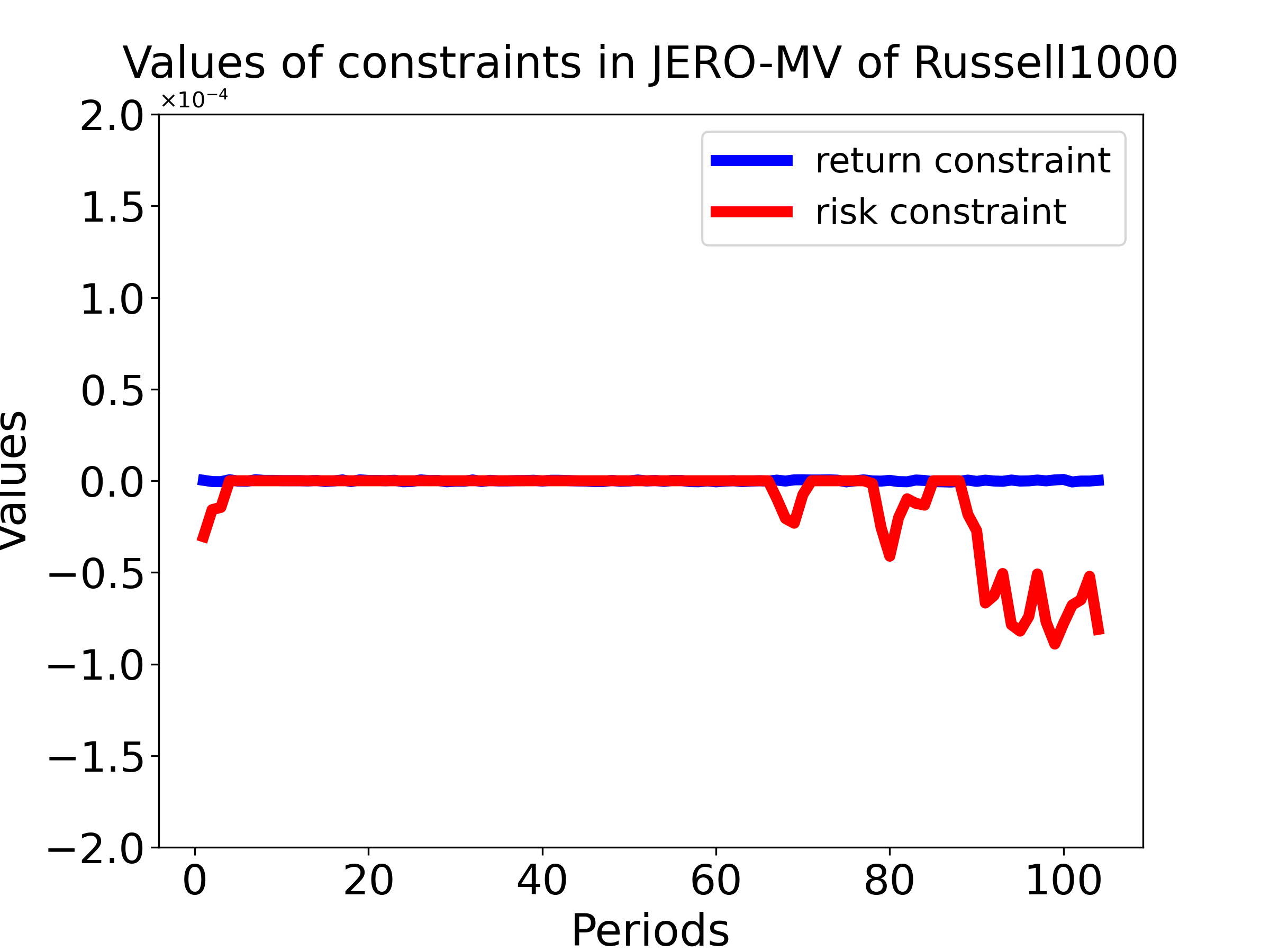}
		\subcaption{}
	 \end{subfigure}
     \begin{subfigure}{0.45\linewidth}
		\includegraphics[width= 1\textwidth]{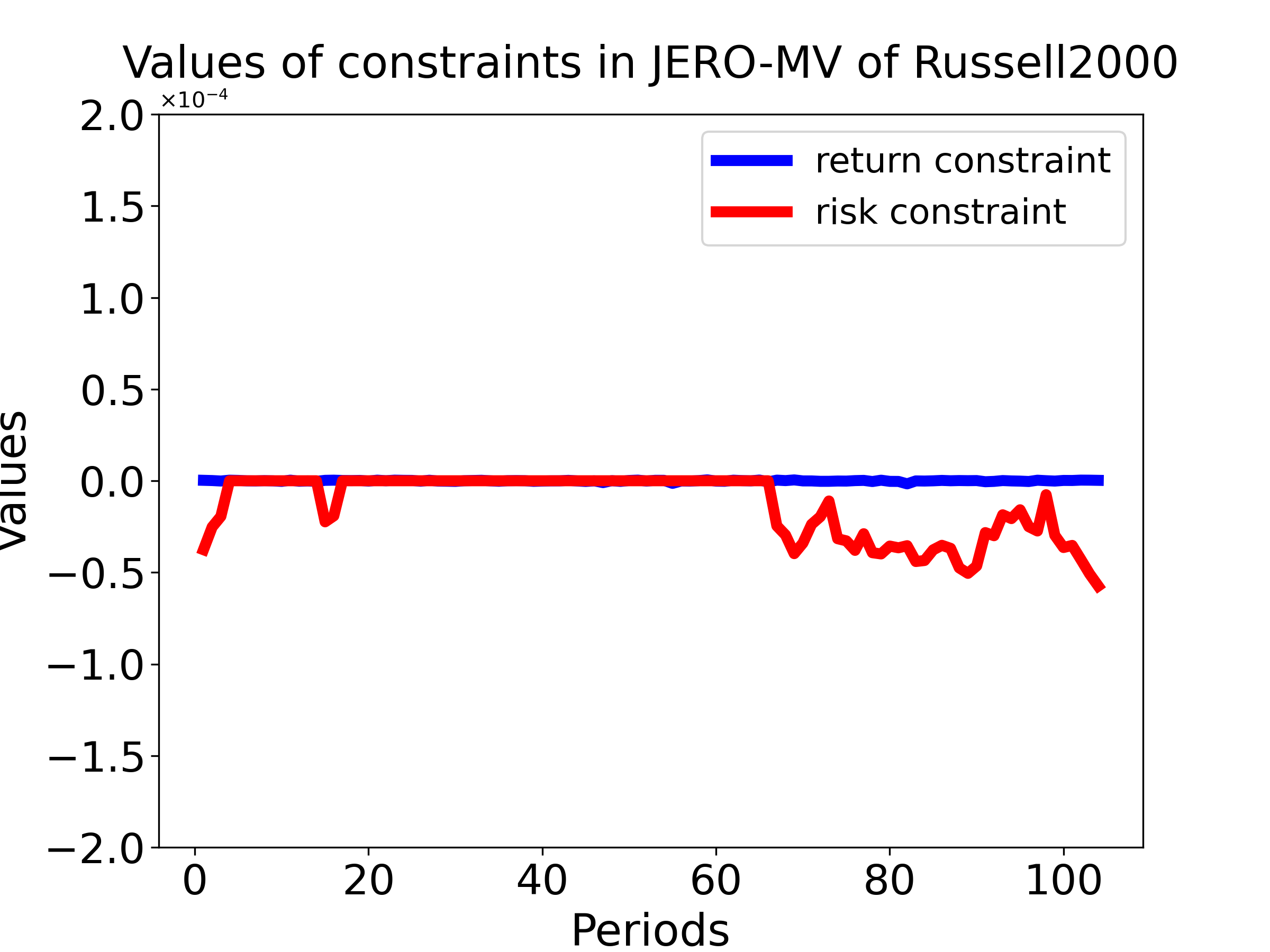}
		\subcaption{}
	 \end{subfigure}
  \caption{Values of the robust return and risk constraints in the JERO-MV optimal portfolios for data sets in Table \ref{Tab-dataset}. (a) - (f) demonstrate the values of the robust return constraint (blue curve) and robust risk constraint (red curve).}
  \label{Fig-constraints}
\end{figure}

Fig. \ref{Fig-constraints} shows the values of the robust return and risk constraints in JERO-MV model for the six data sets, respectively. 
For each data set, we have two sets consisting of $Q=104$ values of the robust return constraint and $Q=104$ values of the robust risk constraint.
The blue curve demonstrates the values of the robust return constraint. We observe that, for the six data sets, all values of the robust return constraint are almost $0$.
The red curve demonstrates the values of the robust risk constraint in JERO-MV model. 
We see that the values are not always $0$, especially those of HS and NDX100.  
Therefore, the robust return constraint in JERO-MV is more likely active, while the robust risk constraint is not always active. When both constraints are active, the return and the risk of the optimal portfolio $\bfw^\ast$ are respectively given by
\[
  \sigma_\ast^2 = \frac{\sigma^2_\target}{e^{2r_\ast}}, \quad \mu_\ast = \mu_\target + \sqrt{e^{2r_\ast}-1}\sigma_\ast.
\]
The smaller $r_\ast$ is, the closer the return and risk of the optimal portfolio will be to their respective targets.

\begin{figure}[!h]
	\centering
	\includegraphics[width= 0.5\textwidth]{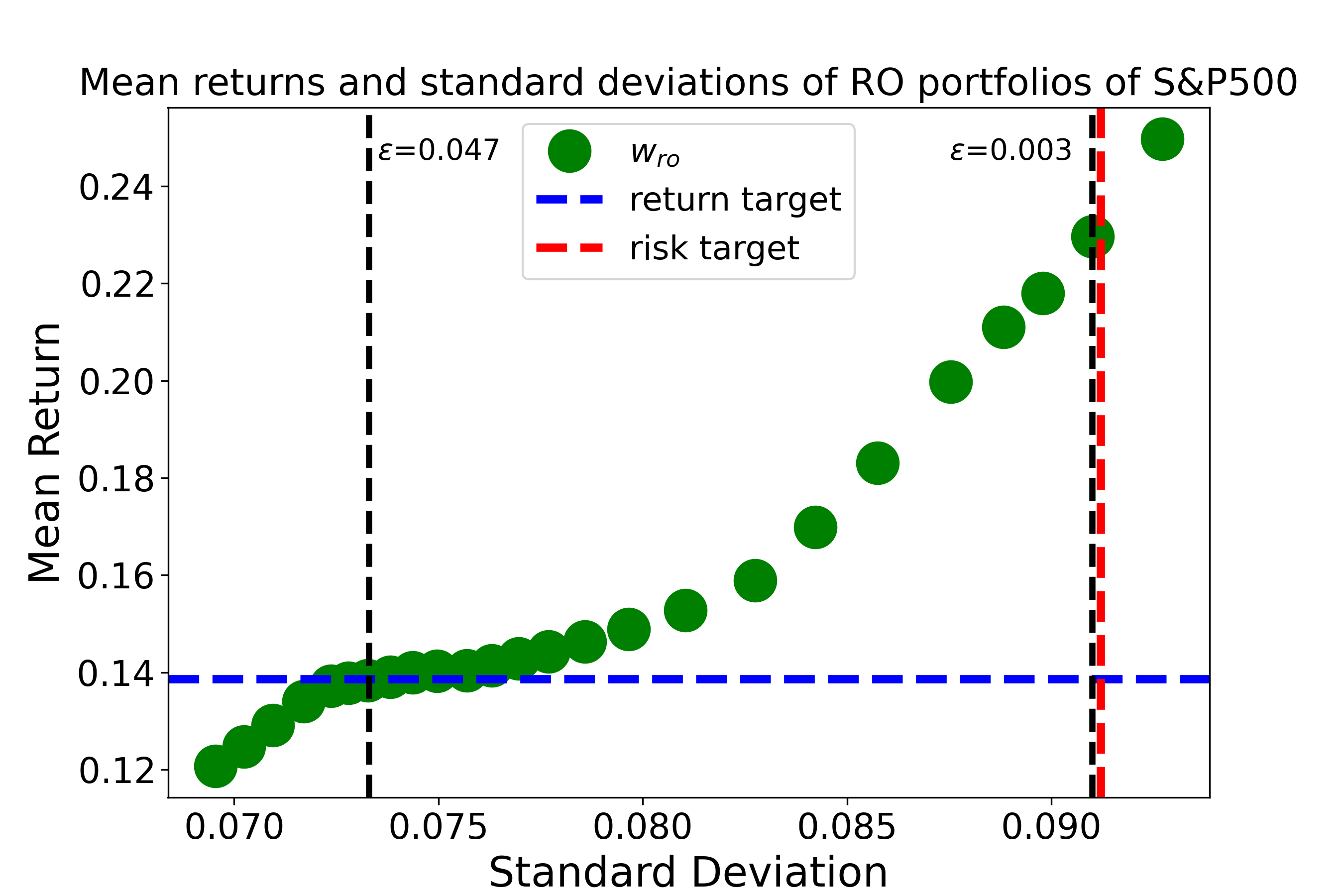}
    \caption{Out-of-sample mean returns and standard deviations of RO portfolios of S$\&$P500. As the size $\epsilon$ of the uncertainty set changes, we compute the corresponding RO portfolio $\bfw_\r$. The two black dotted lines demonstrate $\epsilon = 0.047$ and $\epsilon = 0.003$, respectively. They are the threshold at which out-of-sample performances begin to fail to meet the return target or risk target.}
    \label{Fig-ro}
\end{figure}

In the experiment, we found that for each data set, the size $r_\ast$ of the uncertainty set in JERO-MV is much smaller than $\epsilon$ in RO (note: $\epsilon$ was pre-set in RO).
As stated in \cite{Bertsimas2018data}, the performances of RO portfolio depend on the size of uncertainty set $\epsilon$.
Take S$\&$P500 as an example, Fig. \ref{Fig-ro} shows the out-of-sample mean returns and risks of RO portfolios with different $\epsilon$. 
Specifically, with $\epsilon = 0.001, 0.003, \cdots, 0.05$, we calculate the out-of-sample annualized mean return and risk corresponding to different $\epsilon$ by rolling window method. 
In Fig. \ref{Fig-ro}, the green points represent the RO portfolios. 
The return and risk targets are indicated by blue and red dotted lines, respectively. 
From left to right, $\epsilon$ decreases gradually, that is, the $\epsilon$ corresponding to the leftmost green dot is $0.05$, and the $\epsilon$ corresponding to the rightmost green dot is $0.001$. 
We found that when $\epsilon \in [0.003, 0.047]$, the out-of-sample performances of RO portfolios meet the return and risk targets. 
However, when we set a smaller $\epsilon$ than $0.003$, the out-of-sample risk will exceed the risk target. 
On the contrary, when we set a larger $\epsilon$ than $0.047$, the out-of-sample return will be lower than the return target. 
However, there is no uniform rule for the selection of the uncertainty set size \citep{Kang2019data, Butler2021data}. 
Next, we investigate what the out-of-sample performances of RO would be like, when we set $\epsilon = r_\ast$.

{\bf (D3) JERO-MV can help RO.} We consider the particular choice $\epsilon = r_\ast$. Table \ref{Tab-RO-r} reports the out-of-sample performances of RO portfolios. 
Compared with the performances of RO portfolios in Table \ref{Tab-oos}, we observe that for HS, NDX100, S$\&$P500, Russell1000 and Russell2000 data sets, the out-of-sample mean returns and Sharpe ratios of RO portfolios with $\epsilon = r_\ast$ are higher than those of RO portfolios with $\epsilon = \chi^2_{\alpha,N}$ and $\alpha = 95$. 
Although the out-of-sample standard deviations of RO portfolios with $\epsilon = r_\ast$ are higher, they still meet the given risk targets. 
Therefore, when $\epsilon = r_\ast$, RO portfolios not only meet the return and risk targets but also have better out-of-sample performances for those data sets.

\begin{table}[!h]
  \tabcolsep=0.2cm 
  \centering
  \caption{Out-of-sample performances of RO portfolios when $\epsilon = r^\ast$}\label{Tab-RO-r}
 \begin{tabular}{cccccccc}
  \hline
   ~ & ~ & DAX & HS & NDX100 & S$\&$P500 & Russell1000 & Russell2000\\[1 ex]
   ~ & ~ & N=27 & N=67 & N=90 & N=485 & N=859 & N=1214\\
  \hline
  \multirow{4}*{RO}       & MR & 0.0090  & 0.2032 & 0.1668 & 0.2278 & 0.2335 & 0.0996 \\
                      ~   & SR & -0.2283 & 1.9701 & 1.5829 & 2.1273 & 2.2294 & 0.8638 \\
                      ~   & SD & 0.0416  & 0.0866 & 0.0835 & 0.0908 & 0.0892 & 0.0752 \\
                      ~   & SI & 0.0028  & 0.0020 & 0.0021 & 0.0005 & 0.0003 & 0.0022 \\
  \hline
 \end{tabular}
\end{table}

\section{Conclusion} \label{Section-Conclusion} 

JERO is a new and promising framework for robust optimization. 
Previously, it was applied to portfolio construction with a robust mean-return constraint, leading to JERO-M model. 
In this paper, we study the case when the robust risk constraint is also included, leading to JERO-MV model. 
Compared with JERO-M, JERO-MV cannot be cast as convex-concave framework that allows a tractable representation of the robust mean-return constraint. 
A tractable representation of the robust risk constraint was obtained through S-Lemma.

Although JERO-MV does not optimize return or risk of a portfolio directly, it is surprising to prove that it acts like a utility function that selects an efficient portfolio. 
Moreover, the efficient portfolio has a closed-form formula when JERO-MV has only the budget constraint. 
When more constraints were included in JERO-MV, the feasibility issue arises as the two robust constraints (mean vs risk) may conflict with each other.
We successfully addressed this issue by emploting a bisection method.
The numerical results provide a clear picture of how JERO-MV would select portfolios.

This research also paves the way to including other robust constraints in JERO. 
For instance, it is known that the portfolio diversification return (DR) of \cite{booth1992diversification} is closely related to the portfolio risk. 
When DR is above a certain level, the resulting portfolio often shows outstanding out-of-sample performance
\citep{carmichael2018rao, maeso2020maximizing, qi2023geometric}.
Hence, it is desirable to robustify the DR constraint in JERO and investigate its effect. 
Another worthy research topic is to employ machine learning algorithms would allow the models to include many more constraints and enlarge the application venues. By integrating learning with contextual optimization, we can enhance the overall performance of the portfolio \cite{Sadana2025a}.
We plan to pursuit the investigation along those lines in our future work.

\begin{appendices}

\renewcommand{\theequation}{A.\arabic{equation}} 
\setcounter{equation}{0} 
\section*{Appendix A \quad Proof of Proposition \ref{Thm-variance}}\label{secA1}

We first recall the S-procedure about two quadratic functions. Suppose there are two symmetric matrices $A_1, A_2 \in \mathbb{S}^n$, $\bfb_1, \bfb_2 \in \Re^n$,
and $c_1, c_2 \in \Re$. Let
\[
 F_i(\bfx) := \bfx^\top A_i \bfx + 2 \bfb_i^\top \bfx + c_i, \quad i =1, 2.
\]

\begin{lemma} \cite[Appendix B.2]{boyd2004convex}
 \label{Lemma-S}
 Assume that there exists $\hat{\bfx} \in \Re^n$ such that $F_1(\hat{\bfx}) <0$.
The implication
\[
   F_1(\bfx) \le 0 \qquad \Longrightarrow \qquad F_2(\bfx) \le 0 
\]	
holds if and only if there exists $\lambda \ge 0$ such that
\[
 \left[
  \begin{array}{ll}
  	  A_2 & \bfb_2 \\
  	  \bfb_2^\top & c_2 
  \end{array} 
 \right] \preceq  \lambda \left[ 
 \begin{array}{ll}
 	A_1 & \bfb_1 \\
 	\bfb_1^\top & c_1
 \end{array} 
 \right ] .
\]

\end{lemma} 
{\bf Proof of Proposition \ref{Thm-variance}.} The key observation in this proof is that the robust risk constraint (10) can be restated as the implications of
two quadratic functions so that the S-procedure above can be applied.
We only consider the case $r>0$. Otherwise, it is trivial since $\E(0; \D)$ would contain a single point and the result holds automatically. 
To simplify our proof, let us make some variable substitution. Let
\[
d := \sqrt{ e^{2r} -1} > 0 \qquad \mbox{and} \qquad \bfy := \frac{1}d (\bfbeta - \hatbeta).
\]
Then $\bfbeta \in \E(r; \D)$ if and only if $\bfy^\top \hatC^{-1} \bfy \le 1$. Therefore, 
the robust risk constraint
(10) holds if and only if the following implication involving two quadratic functions in $\bfy$ holds
\be \label{Implication-A}
 \bal
 F_1(\bfy) := \bfy^\top \hatC^{-1} \bfy -1 \le 0  \quad
\Rightarrow \quad
F_2(\bfy) := d^2 \bfy^\top \bfw \bfw^\top \bfy + \bfw^\top \hatC\bfw - \sigma_\target^2 \le 0.
 \eal
\ee
Note that $\bfy=0$ satisfies $\bfy^\top \hatC^{-1} \bfy -1 = -1 < 0$.
By the S-Procedure, the implication (\ref{Implication-A}) holds if and only if
there must exists $\lambda \ge 0$ such that (we use $\sigma^2$ for $\sigma^2_\target$ for simplicity)
\[
\left[
\begin{array}{cc}
	d^2 \bfw \bfw^\top & 0\\
	0  & \bfw^\top \hatC \bfw - \sigma^2
\end{array} 
\right] \preceq \lambda \left[
\begin{array}{cc}
	\hatC^{-1} \bfw \bfw^\top & 0\\
	0  & -1 
\end{array} 
\right] .
\]
In other words, 
\[
M := \left[
\begin{array}{cc}
	\lambda \hatC^{-1} -	d^2 \bfw \bfw^\top & 0\\
	0  & \sigma^2 -\lambda -  \bfw^\top \hatC \bfw 
\end{array} 
\right] \succeq 0 .
\]
Now we let
\[
H := \hatC^{1/2} \bfw \bfw^\top \hatC^{1/2} = Q \Lambda Q^\top,
\]
where $\Lambda := \diag(\lambda_1, \ldots, \lambda_n)$, with $\lambda_1 \ge \lambda_2 \ge \cdots \ge \lambda_n$ being the
eigenvalues of $H$ in nonincreasing order, and $Q^\top Q = I$ (i.e, $Q$ is the eigenvector matrix).
Define
\[
B := \left[
\begin{array}{cc}
	Q^\top \hatC^{1/2} & 0 \\
	0 & 1
\end{array} 
\right].
\]
Apparently, $B$ is nonsingular. Therefore, $M \succeq 0$ if and only if
\bes
 \bal
0 \preceq BMB^\top 
&= \left[
\begin{array}{cc}
	\lambda I - d^2 Q^\top H Q  & 0\\
	0  & \sigma^2 -\lambda -  \bfw^\top \hatC \bfw 
\end{array} 
\right]  = \left[
\begin{array}{cc}
	\lambda - d^2 \Lambda  & 0\\
	0  & \sigma^2 -\lambda -  \bfw^\top \hatC \bfw 
\end{array} 
\right] .
 \eal
\ees
This is equivalent to
\be \label{Condition-1}
\lambda I - d^2 \Lambda \succeq 0 \qquad \mbox{and} \qquad
\sigma^2 -\lambda -  \bfw^\top \hatC \bfw  \ge 0 .
\ee 
We now examine the first condition in (\ref{Condition-1}). It means
\be \label{Condition-2}
d^2 \lambda_1 \le \lambda \quad (\mbox{ $\lambda_1$ is the largest eigenvalue of $H$}).
\ee 
We note that $H$ is rank-1 positive semidefinite matrix. Hence
\[
\lambda_1 = \mbox{Trace}(H) = \mbox{Trace}(\hatC^{1/2} \bfw \bfw^\top \hatC^{1/2})
= \bfw^\top \hatC \bfw ,
\]
which with (\ref{Condition-1}) and  (\ref{Condition-2}) yields 
\[
\lambda \ge (e^{2r} -1) \bfw^\top \hatC \bfw \quad \mbox{and} \quad 
\lambda \le \sigma^2 - \bfw^\top \hatC \bfw. 
\]
This is equivalent to
(11). \hfill $\Box$

\renewcommand{\theequation}{B.\arabic{equation}} 
\setcounter{equation}{0} 
\section*{Appendix B \quad Proof of Proposition \ref{Prop-Sharpe}}

We use the result in Proposition \ref{Prop-EF}. Let the efficient portfolio $\bfw_\gamma$ be parameterized by $\gamma \ge 0$. 
Then the Sharpe ratio can be represented as
\begin{equation*}
\begin{aligned}
    f(\gamma) :=  \  \frac{\mu_\gamma- \mu_{target}  }{ \sigma_\gamma }
 = \  \frac{\mu_{mvp}- \mu_{target} + \gamma \mu_0 }{ \sqrt{\gamma^2 + \sigma_{mvp}^2 } }
\end{aligned}
\end{equation*}
It is differentiable with respect to $\gamma$. Simple calculation shows that its derivative is given by
\[
 f'(\gamma) = \frac{\mu_0 \sigma_\mvp^2 + \gamma(\mu_\target - \mu_\mvp) }{ (\gamma^2 + \sigma_\mvp^2 )^{3/2} }.
\]
Since $\mu_0 \ge 0$, we have $f'(\gamma) \ge 0$ under the assumption $\mu_\target \ge \mu_\mvp$. 
Therefore, the Sharpe ration is an increasing function and reaches
its maximum at $\gamma$ when $\sigma_\gamma$ is the largest. Therefore, Sharpe portfolio is the efficient portfolio that realizes the risk target
$\sigma_\target$. This proves the result.  \hfill $\Box$

\renewcommand{\theequation}{C.\arabic{equation}} 
\setcounter{equation}{0} 
\section*{Appendix C \quad Proof of Proposition \ref{Prop-r-mv}}

 (i) is easy to prove. We note that the function $r_v(\bfw)$ only depends on the variance of $\bfw$ and
\[
  r_v(\bfw) = \frac 12 \ln(\sigma_\target^2) - \frac 12 \ln(\sigma_\bfw^2).
\]
The result follows from the fact that the function $\ln(\cdot)$ is increasing.

The first part in (ii) follows from the following inequality:
\begin{equation*}
 \begin{aligned}
    r_m(\bfw_\ef^\sigma) = \ \frac 12 \ln \Big(
	1 +  (\hatbeta^\top \bfw_\ef^\sigma - \bfmu_\target)^2 / \sigma^2
	\Big) \ge \  \frac 12 \ln \Big(
	1 +  (\hatbeta^\top \bfw - \bfmu_\target)^2 / \sigma_{\bfw}^2
	\Big)  = \ r_{(\bfw)},
 \end{aligned}
\end{equation*}
where the inequality used the fact $\hatbeta^\top \bfw_\ef^\sigma \ge \hatbeta^\top \bfw$ (due to $\bfw_\ef^\sigma$ and $\bfw$ shares the same variance, and $\bfw^\sigma_\ef$ is efficient).

The second part in (ii) is due to the following observation.  
The function $r_m(\bfw_\ef^\sigma)$ is an increasing function of the Sharpe ratio (13) and the Sharpe ratio is an increasing
function along the efficient frontier upto Sharpe portfolio (see the proof of Proposition 3). 
Therefore, $r_m(\bfw_\ef)$ is an increasing function over the interval
$[\sigma_{\mu target}, \ \sigma_{target}]$ and it reaches its maximum at $\bfw_{\sigma target}$. Moreover, 
\begin{equation*}
 \begin{aligned}
   r_m(\bfw_{\mu \target}) = &  \frac 12 \ln \Big(
    (\hatbeta^\top \bfw_{\mu target} - \bfmu_{target})^2 / \sigma^2 + 1   \Big) = \frac 12 \ln 1 = 0, 
 \end{aligned}
\end{equation*}
where the second equation is due to $\hatbeta^\top \bfw_{\mu\target} = \bfmu_\target$.

(iii) The result in (i) says that $r_v(\bfw_\ef^\sigma)$ is decreasing when $\sigma$ changes from $\sigma_{\mu\target}$ to $\sigma_\target$ with the end value
$r_v(\bfw_{\sigma\target}) = 0$.
The result in (ii) says that $r_m(\bfw_\ef^\sigma)$ is increasing as $\sigma$ varies from  $\sigma_{\mu\target}$ to $\sigma_\target$ with
the starting value $r_m(\bfw_{\mu\target}) =0$.
Therefore,
the two quantities $r_m(\bfw_\ef)$ and $r_v(\bfw_\ef)$ must become equal at some point.
On the other hand, (ii) implies that  $r_m(\bfw)$ achieves its maximum at $\bfw_\ef^\sigma$ over all $\bfw \in \W_\sigma$.
It follows from (14) that the optimal $r_*$ is achieved when  $r_m(\bfw_\ef)$ and $r_v(\bfw_\ef)$ become equal. 
This finishes the proof. \hfill $\Box$

\renewcommand{\theequation}{D.\arabic{equation}} 
\setcounter{equation}{0} 
\section*{Appendix D \quad Proof of Theorem \ref{Thm-EF-A}}

The optimal solution $\bfw$ of JERO-MV model (12) is decided by the equation (15), which is equivalent to
\be \label{Eq-S}
  \sigma_\target^2 = \sigma_\bfw^2 + \Big( \hatbeta^\top \bfw - \mu_\target \Big)^2 .
\ee
Since $\bfw$ is efficient as established above, it satisfies the efficient frontier equations for some $\gamma>0$:
\bes 
 \bal
  & \bfw = \bfw_\mvp + \gamma_\target \bfw_0, \ \hatbeta^\top \bfw = \mu_\mvp + \gamma \mu_0, \ \mbox{and} \ \sigma_\bfw^2 = \sigma_\mvp^2 + \gamma^2 .
 \eal
\ees
Substitute those into (\ref{Eq-S}) to get
\be \label{Eq-rho}
 \bal
  & (1+\mu_0^2) \gamma^2  - 2\mu_0 (\mu_\target - \mu_\mvp ) \gamma +  (\mu_\target - \mu_\mvp )^2 + \sigma_\mvp^2 - \sigma_\target^2 = 0 .
 \eal
\ee
This is a quadratic equation in $\gamma$ and its discriminant term is
\bes
 \bal
\Delta_1 := & \ 4 \mu_0^2 	(\mu_\target - \mu_\mvp )^2 - 4 (1+\mu_0^2) \left[ \sigma_\mvp^2 - \sigma_\target^2  + (\mu_\target - \mu_\mvp )^2 \right] \\
= & \ -4 (\mu_\target - \mu_\mvp )^2 + 4 (1+\mu_0^2) (\sigma_\target^2 - \sigma_\mvp^2 ) \\
= & \ 4 (\sigma_\target^2 - \sigma_\mvp^2 ) + 4 \mu_0^2 
   \frac{(\mu_{\sigma\target} - \mu_\mvp)^2 }{\mu_0^2} -4 (\mu_\target - \mu_\mvp )^2 \\
= & \ 4 (\sigma_\target^2 - \sigma_\mvp^2 ) + 4 \left[
(\mu_{\sigma\target} - \mu_\mvp)^2 - (\mu_\target - \mu_\mvp )^2 \right] \\
= & \ 4(\sigma_\target^2 - \sigma_\mvp^2) + 4(\mu_{\sigma\target} - \mu_\target)(\mu_{\sigma\target} + \mu_\target - 2\mu_\mvp) \\
= & \ 4 \Delta .
 \eal
\ees
Since we assumed $\mu_{\sigma\target} \ge \mu_\target \ge \mu_\mvp$, we have
\bes
 \bal
  \Delta_1 \ge & \ 4(\sigma_\target^2 - \sigma_\mvp^2) = \ 4 \mu_0^2 ( \mu_{\sigma\target} - \mu_\mvp  )^2 \geq \ 4 \mu_0^2 ( \mu_{\target} - \mu_\mvp  )^2,
 \eal
\ees
where the equality used the fact in Proposition \ref{Prop-EF} (ii). The consequence of this bound for $\Delta_1$ means
that the quadratic equation (\ref{Eq-rho}) has only one positive root for $\gamma$, which is given by
(16).
The optimal portfolio for JERO-MV model (12) can be represented as claimed in the result.
\hfill $\Box$ 

\renewcommand{\theequation}{E.\arabic{equation}} 
\setcounter{equation}{0} 
\section*{Appendix E \quad Proof of Theorem \ref{Thm-EF-B}}

We first note that any feasible portfolio must have its variance $\sigma^2 \in [\sigma_{\mu\target}^2, \ \sigma_\target^2]$.
Furthermore, it is easy to verify the following equivalence.
\be \label{Bound-1}
 \bal
	& r_v(\bfw_\mvp) \le r_m(\bfw_\mvp) \quad \\
	\Longleftrightarrow  \quad & \frac{\sigma_\target^2}{\sigma_\mvp^2} \le 1 + \frac{(\mu_\mvp - \mu_\target)^2}{\sigma_\mvp^2} \quad \\
	\Longleftrightarrow  \quad & \sigma_\target^2 \le \sigma_\ell^2 .
 \eal
\ee
We now use the facts established in Proposition 4.
The first fact is that $r_v(\bfw_\ef^\sigma)$ is decreasing when $\sigma \in [\sigma_{\mu\target}, \ \sigma_\target]$ and
$r_v(\bfw_\mvp)$ is the largest. 
The second fact is that $r_m(\bfw_\ef^\sigma)$ is increasing from $\bfw_\mvp$ and reaches its maximum at $\bfw_T$ by the definition of the tangent
portfolio.
Those two facts with 
 the assumption $\sigma_\target \le \sigma_\ell$ imply that
\[
r_m(\bfw_\ef^\sigma) \ge r_v(\bfw_\ef^\sigma)
\]
for $\bfw_\ef^\sigma$ moving from $\bfw_\mvp$ to $\bfw_T$.
Therefore, the optima value $r_*$ is given by
\bes
 \bal
r_* = & \max\{ \min\{ r_m(\bfw_\ef^\sigma), \  r_v(\bfw_\ef^\sigma) \} \ | \ \bfw_\ef^\sigma \ \mbox{is feasible}  \}	\\ 
    \le & \max\{ r_v (\bfw_\ef^\sigma) \ | \ \bfw_\ef^\sigma\ \mbox{is feasible} \} \\
    = & \ r_v(\bfw_\mvp) 
    =  \max\{ \min\{ r_m(\bfw_\mvp), \; r_v(\bfw_\mvp) \} ,	
 \eal
\ees
where the last equality used (\ref{Bound-1}).
Hence, $\bfw_\mvp$ yields the largest $r_*$ and it is selected by JERO-MV model.

Now suppose $\sigma_\target \ge \sigma_u$. This means that $\bfw_T$ is feasible. 
It is easy to see that
\be \label{Bound-2}
 \bal
	& r_v(\bfw_T) \ge r_m(\bfw_T) \quad \\
	\Longleftrightarrow \quad & \frac{\sigma_\target^2}{\sigma_T^2} \ge 1 + \frac{(\mu_T - \mu_\target)^2}{\sigma_T^2} \quad \\
	\Longleftrightarrow \quad & \sigma_\target^2 \ge \sigma_u^2 .
 \eal
\ee
Since $r_m(\bfw_\ef^\sigma)$ is increasing from $\bfw_\mvp$ to $\bfw_T$ and $r_v(\bfw_\ef^\sigma)$ is decreasing from $\bfw_\mvp$ to
$\bfw_{\sigma\target}$, and $r_m (\bfw_T)$ is the largest value that $r_m(\bfw)$ can achieve, we have
$ r_v(\bfw^\sigma_\ef) \ge r_m(\bfw_\ef^\sigma)$ for $\bfw$ moving from $\bfw_\mvp$ to $\bfw_T$.
 We have
\bes
 \bal
	r_* = & \max\{ \min\{ r_m(\bfw_\ef^\sigma), \ r_v(\bfw_\ef^\sigma) \} \ | \ \bfw_\ef^\sigma \ \mbox{is feasible}  \}	\\
    \le & \max\{ r_m (\bfw_\ef^\sigma) \ | \ \bfw_\ef^\sigma\ \mbox{is feasible} \} \\
    = & \ r_m(\bfw_T) 
    =  \max\{ \min\{ r_m(\bfw_T), \; r_v(\bfw_T) \} ,	
 \eal
\ees
where the last equality used (\ref{Bound-2}).
Hence, $\bfw_T$ yields the largest $r_*$ and it is selected by JERO-MV model.

For the remaining case $\sigma_\target \in (\sigma_\ell, \ \sigma_u)$, the optimal point is defined by the optimality equation 
(15). We arrive at the quadratic equation (\ref{Eq-rho}).
Its discriminant term is
\bes
 \bal
	\Delta_1 
	=& \ -4 (\mu_\target - \mu_\mvp )^2 + 4 (1+\mu_0^2) (\sigma_\target^2 - \sigma_\mvp^2 ) \\
	=& \ 4 \left[
	\sigma_\target^2 - \sigma_\mvp^2  - (\mu_\target - \mu_\mvp )^2
	\right] + 4 \mu_0^2 (\sigma_\target^2 - \sigma_\mvp^2 ) \\
	=& \ 4 ( \sigma_\target^2 - \sigma_\ell^2) +  4 \mu_0^2 (\sigma_\target^2 - \sigma_\mvp^2 ) \\
	\ge & \ 4 \mu_0^2 (\sigma_\target^2 - \sigma_\mvp^2 ) \\
    =& \ 4 \mu_0^2 (\mu_\target - \mu_\mvp)^2 .
 \eal
\ees
Hence, the quadratic equation (\ref{Eq-rho}) has only one positive root and it is given by (16).
The three cases in this result are plotted in Fig.~2.
\hfill $\Box$
\end{appendices}

\bibliography{referencesjeromv}

\end{document}